\documentclass[a4paper,12pt]{article}
\usepackage{geometry,latexsym,amssymb}
\usepackage[latin5]{inputenc}
\usepackage{enumerate}
\usepackage{enumitem}
\usepackage[usenames]{color}
\usepackage{graphicx}
\usepackage{enumerate}
\usepackage{rotating}
\usepackage{multirow}
\usepackage{cite}
\usepackage{amsthm}
\usepackage{caption}
\usepackage{tikz}
\usetikzlibrary{matrix,arrows}
\usepackage{amsmath,amscd}
\usepackage{calc}
\usepackage{tabularx}
\usepackage{ifthen}

\usepackage{amsfonts}
\usepackage{amsmath}
\usepackage{footnote}
\usepackage{multicol}
\usepackage{booktabs}
\usepackage[flushleft]{threeparttable}
\usepackage[font=small,labelfont=bf]{caption}
\usepackage{lipsum}
\usepackage{eucal}
\usepackage{amssymb,amscd,amsthm, verbatim,amsmath,color,fancyhdr, mathrsfs}
\usepackage{graphicx}
\usepackage{turnstile}
\usepackage[colorinlistoftodos]{todonotes}

\newcommand\blfootnote[1]{%
	\begingroup
	\renewcommand\thefootnote{}\footnote{#1}%
	\addtocounter{footnote}{-1}%
	\endgroup
}
\newtheorem{thm}{Theorem}[section]
\newtheorem{prop}[thm]{Proposition}
\newtheorem{lem}[thm]{Lemma}

\newtheorem{cor}[thm]{Corollary}
\newtheorem{conj}[thm]{Conjecture}

\newtheorem{ex}[thm]{Example}

\date{\today}
	
\begin{document}
		
	\begin{center}
		{\Large \textbf{Signed Mahonian Polynomials on Colored Derangements}} \vspace*{0.5cm}
	\end{center}
	
	\begin{center}
		Hasan Arslan$^{*,a}$, Moussa Ahmia$^{b}$, Nazmiye Alemdar$^{a}$ 
	\end{center}
	\begin{center}
		$^{a}${\small {\textit{Department of  Mathematics, Erciyes University, 38039, Kayseri, Turkey}}}\\
		$^{b}${\small {\textit{Department of  Mathematics, Mohamed Seddik Benuyahia University, LMAM laboratory, Jijel, Algeria}}}\\
	\end{center}
	
	\blfootnote{Email Adressess: hasanarslan@erciyes.edu.tr (H. Arslan),  moussa.ahmia@univ-jijel.dz (M. Ahmia), nakari@erciyes.edu.tr (N. Alemdar)\\
		*Corresponding Author: Hasan Arslan}

		\begin{abstract}
The polynomial $\sum_{\pi \in W}q^{maj(\pi)}$ of major index over a classical Weyl group $W$ with a generating set $S$ is called the Mahonian polynomial over $W$, and also the polynomial $\sum_{\pi \in W}(-1)^{l(\pi)}q^{maj(\pi)}$ of major index together with sign over the group $W$ is called the signed Mahonian polynomial over the group $W$, where $l$ is the length function on $W$ defined in terms of the generating set $S$. We concern with the signed Mahonian polynomial 
$$\sum_{\pi \in D_{n}^{(c)}}(-1)^{L(\pi)}q^{fmaj(\pi)}$$ 
on the set $D_{n}^{(c)}$ of colored derangements in the group $G_{c,n}$ of colored permutations, where $L$ denotes the length function defined by means of a complex root system described by Bremke and Malle in $G_{c,n}$ and $fmaj$ defined by Adin and Roichman in $G_{c,n}$ represents the \textit{flag-major index}, which is a Mahonian statistic. As an application of the formula for signed Mahonian polynomials on the set of colored derangements, we will derive a formula to count colored derangements of even length in $G_{c,n}$ when $c$ is an even number. Finally, we conclude by providing a formula for the difference between the number of derangements of even and odd lengths in $G_{c,n}$ for every positive integer $c$, regardless of whether c is odd or even.
			
		\end{abstract}


\\

		\section{Introduction}
	In this paper, we deal with the $q$-enumeration of the Mahonian polynomial of colored derangements in the group of colored permutations $G_{c,n}$ based on the flag major index. Let $\mathbb{N}$ be the set of natural numbers. 
For $n\in\mathbb{N}$, let $[n]:=\{1,2,\ldots,n\}$ (where $[0]:=\emptyset$). 
The cardinality of a set $X$ will be denoted by $|X|$.
Let $c$ and $n$ be positive integers. The group of colored permutations $G_{c,n}$ may be described as follows:
The group $G_{c,n}$ of $n$ letters with $c$ colors may be regarded as a split group extension
$G_{c,n}=\mathfrak{S}_n \rtimes  \mathbb{Z}_c^n$ of $\mathbb{Z}_{c}^n$ by $\mathfrak{S}_n$, 
where $\mathbb{Z}_{c}:=\{0,1,\ldots,c-1\}$, $\mathbb{Z}_{c}^n$ is the direct product of $n$ copies of 
$\mathbb{Z}_c$ and $\mathfrak{S}_n$ is the symmetric group on $[n]$.
A much more natural manner to introduce $G_{c,n}$ is in the following way: 
Now consider the set $I=\{ i^{[k]}~:~i=1, \ldots, n; ~k= 0, \ldots, c-1\}$ as the set $[n]$ colored by the colors $0,1,\ldots, c-1$.
Then, the colored permutations group $G_{c,n}$ 
consists of all bijections $\pi~:~I \mapsto I$ providing that $\pi(i^{[k]} )=(\pi(i))^{[k]}$ 
for all $i=1, \ldots, n; ~k=0, \ldots,c-1$. 
Therefore, any element $\pi$ of the group $G_{c,n}$ is written as follows:
$$
\pi = \bigl(\begin{smallmatrix}
	1 & 2 &  \cdots  & n \\
	\pi_1^{[t_1]} & ~\pi_2^{[t_2]}& \cdots   &  \pi_n^{[t_n]}
\end{smallmatrix}\bigr),
$$
where $\pi_i\in [n]$ and $0\leq t_i \leq c-1$ for all $i=1,\ldots,n$. We shall represent $\pi$ by the word by removing its first row:
$$\pi=\pi_1^{[t_1]} ~\pi_2^{[t_2]} \cdots  \pi_n^{[t_n]}.$$
The \textit{color} number of $\pi$ is defined as the sum $\textrm{col}(\pi)=\sum_{i=1}^n t_{i}$. In particular, the underlying permutation of a colored permutation $\pi=\pi_1^{[t_1]} ~\pi_2^{[t_2]} \cdots  \pi_n^{[t_n]} \in G_{c,n}$
is denoted by $|\pi|=\pi_1 ~\pi_2 \cdots  \pi_n \in \mathfrak{S}_n$.

%In other words, colored permutation group $G_{c,n}$ consists of ordered pairs $(\beta,t)$, where $\beta=\beta_1 \beta_2 \cdots \beta_n \in S_n$ and $t=(t_1,\cdots,t_n)$ is an $n$-tuple of integers with all $t_i,~i=1,\cdots,n$ between $0$ and $m-1$. An element $(\beta, t)\in G_{c,n}$ will be represented as a word 
%\begin{equation*}
%	\beta_1^{[t_1]}\beta_2^{[t_2]}\cdots %\beta_n^{[t_n]}
%\end{equation*}
%such that we remove the superscript $[t_i]$ when $t_i=0$. It is called to be the \textit{window notation} of $(\beta, t)$. For example, $(7136425,~(3,1,2,0,3,0,4))\in G_{5,7}$ can be represented as $7^{[3]}1^{[1]}3^{[2]}64^{[3]}25^{[4]}$. Clearly, $|G_{c,n}|=c^nn!$. The multiplication in $G_{c,n}$ is defined as follows: For any $(\beta, t)$ and $(\sigma, t')$ such that $t=(t_1,\cdots,t_n)$ and $t'=(t'_1,\cdots,t'_n)$ in $\mathbb{Z}_c^n$
%\begin{equation*}
%	(\beta, t)\cdot (\sigma, t')=(\beta \circ \sigma, (t'_1+t_{\sigma(1)},\cdots, t'_n+t_{\sigma(n)})).
%\end{equation*}
%where $\beta \circ \sigma$ is the composition in $\mathfrak{S}_n$. The group $G_{c,n}$ can be generated by the canonical set of generators
%\begin{equation*}
%	S_{G_{c,n}}=\{r_0,r_1,\cdots,r_{n-1}\},
%\end{equation*}
%defined with their action on the set $[n]$ in the following way:
%$$r_i(j)=
%\begin{cases}
%	i+1,  &  \textrm{if}~ j=i; \\
%	i,  &  \textrm{if}~ j = i+1;\\
%	j,  &  \textrm{otherwise.}
%\end{cases}
%$$ while exceptional generator $r_0$ is identified %with 
%$$r_0(j)=
%\begin{cases}
%	1^{[1]},  & \textrm{if} ~ j=1; \\
%	j,  &  \textrm{otherwise.}
%\end{cases}
%$$

As is well-known from \cite{stanley2011}, the inversion number and the descent set of $\sigma \in \mathfrak{S}_n$ are respectively defined in the following way:
\begin{align*}
inv (\sigma)=&| \{(i,j) \in [n] \times [n] ~:~ i<j~\textrm{and}~\sigma_i > \sigma _{j}\} |, \\
Des (\sigma)=&\{i \in [n-1] ~:~ \sigma_i > \sigma _{i+1}\}. 
\end{align*}

The usual (type $A$) major index, a statistic derivable from $Des (\sigma)$, is identified as $maj (\sigma)= \sum_{i \in Des (\sigma)}i$.
MacMahon proved in \cite{macmahon1915} that the the number of inversions \textit{inv} is equidistributed with the major index \textit{maj} over the symmetric group $\mathfrak{S}_n$, that is, 
$$\sum_{\sigma \in \mathfrak{S}_n}q^{inv(\sigma)}=\sum_{\sigma \in \mathfrak{S}_n}q^{maj(\sigma)}=[n]!_q$$
where $q$ is an indeterminate, $[i]_q:=\frac{1-q^i}{1-q}$ for all $i=1,\ldots,n$ is a $q$-integer and $[n]_q!:=[1]_q [2]_q\cdots [n-1]_q [n]_q$ is the usual $q$-analogue of $n!$.
If a permutation statistic is equidistributed with the length function, that is, the number of inversions, then it is called \textit{Mahonian}.

%Following \cite{adin2001}, we let $\sigma_0:=r_0$ and for all $i \in [1,n-1]$,~$\sigma_i:=r_ir_{i-1}\cdots r_1 r_0 \in B_n$. Thus, the collection $\{\sigma_0, \sigma_1, \cdots, \sigma_{n-1}\}$ is a different set of generators for $G_{c,n}$ and any $\pi\in G_{c,n}$ has a unique expression 
%\begin{equation}\label{flag0}
%\pi=\sigma_{n-1}^{k_{n-1}}\cdots \sigma_{2}^{k_{2}}\sigma_{1}^{k_{1}}\sigma_{0}^{k_{0}}
%\end{equation}
%with $0\leq k_i \leq 2i+1$ for all $0\leq i \leq n-1$. The \textit{flag-major index} was defined for the group $G_{c,n}$ as follows (see \cite{adin2001}): Let $\pi \in G_{c,n}$. Then
%\begin{equation}\label{flag1}
%fmaj(\pi)=\sum_{i=0}^{n-1}k_i.
%\end{equation}

Given $
\pi = \bigl(\begin{smallmatrix}
	1 & 2 &  \cdots  & n \\
	\pi_1^{[t_1]} & ~\pi_2^{[t_2]}& \cdots   &  \pi_n^{[t_n]}
\end{smallmatrix}\bigr) \in G_{c,n}
$, the length of $\pi$ is defined by 
\begin{equation} \label{length}
L(\pi) := \textrm{col}(\pi)+c \cdot \sum_{t_j \neq 0} |\{(i,j) : i<j ~\textrm{and}~ \pi_i<\pi_j\}|+inv(|\pi|)
\end{equation}
due to \cite{bremke1997}, where $inv(|\pi|)$ corresponds to the classical inversion in $\mathfrak{S}_n$. Clearly, $L(|\pi|)=inv(|\pi|)$ for all $\pi \in G_{c,n}$. For example, the length of $\pi=2^{[3]} ~1^{[1]} ~3~4^{[2]}~5 \in G_{4,5}$ is $L(\pi)=6+4 \cdot 3+1=19$. It is well known from \cite{arslan2024} that for each $\pi \in G_{c,n}$ the length $L(\pi)$ corresponds to the colored inversion $inv_c(\pi)$. The flag-major index of any colored element can be practically computed by using the next theorem.

\begin{thm} [Adin-Roichman \cite{adin2001}]\label{adinroichman}
Let $\pi \in G_{c,n}$. Then
\begin{equation}\label{fmaj2001}
fmaj(\pi)=c \cdot maj(\pi)+col(\pi)
\end{equation}
where maj is computed with respect to the following total order: 
\begin{equation}\label{totalord}
1^{[c-1]}<\cdots<n^{[c-1]}<\cdots<1^{[2]}<\cdots<n^{[2]}<1^{[1]}<\cdots<n^{[1]}<1<\cdots < n.
\end{equation}
\end{thm}
In Theorem \ref{adinroichman}, $maj(\pi)$ equals the sum $\sum_{i \in Des_c(\pi)}i$, where $Des_c(\pi):=\{i \in [n-1]: \pi_i > \pi_{i+1}\}$ is determined by considering the total ordering in (\ref{totalord}).

To illustrate Theorem \ref{adinroichman}, we will continue with our running example. Thus, we obtain that the flag-major index of $\pi=2^{[3]} ~1^{[1]} ~3~4^{[2]}~5 \in G_{4,5}$ is $fmaj(\pi)=4 \cdot 3+(3+1+2)=18$. 
Adin and Roichman in \cite{adin2001} showed that the flag major index in the case $c=1$ (i.e., for the symmetric group $\mathfrak{S}_n$) coincides
with the usual major index. It is well known from \cite{arslan2024} that the flag major index  is a Mahonian statistic, that is, 
\begin{equation}\label{eflag}
\sum_{\pi \in G_{c,n}}q^{fmaj(\pi)}=\prod_{i=1}^n [ci]_q=\sum_{\pi \in G_{c,n}} q^{L(\pi)}.
\end{equation}

Bagno and Biagioli \cite{bagno2007} defined flag-major and flag-descent statistics on the complex reflection group $G_{c,p,n}$, which is a normal subgroup of $G_{c,n}$ of index $p$. Then they established their joint distribution on $G_{c,p,n}$, which is called the \textit{Carlitz identity} for classical Weyl group. Although the statistics $inv$ and $maj$ have symmetric joint distribution 
\begin{equation} \label{jointsym}
  \sum_{\pi \in \mathfrak{S}_n}t^{inv(\pi)}q^{maj(\pi)}=\sum_{\pi \in \mathfrak{S}_n}q^{inv(\pi)}t^{maj(\pi)}  
\end{equation}
over $\mathfrak{S}_n$ \cite{foata1978}, the statistics $fmaj$ and $inv_c$ do not generally have a symmetric joint distribution over $G_{c,n}$. For instance, examining Table \ref{2} shows that the joint symmetric distribution condition in question is not satisfied for $G_{4,2}$. In \cite{ahmia2025}, Ahmia, Ramir\'ez and Villamizar gave a combinatorial interpretation of an inversion-like statistic $\widetilde{inv}$ on $G_{c,n}$, which is also known as flag-inversion and appears in Fire \cite{fire2005}, with the help of a decomposition of a colored Lehmer code 
\begin{equation} \label{tileinv}
\widetilde{inv}(\pi)=c \cdot inv(|\pi|)+col(\pi)
\end{equation}
for each $\pi \in G_{c,n}$.  It easily follows from \cite{fire2005} that $\widetilde{inv}$ is equidistributed with $fmaj$ over $G_{c,n}$. In addition, $\widetilde{inv}$ is a Mahonian statistic in terms of the length function $L$. Clearly, $\widetilde{inv}(\pi)=\widetilde{inv}(\pi^{-1})$ for each $\pi \in G_{c,n}$. Furthermore, we will assert in Conjecture \ref{symjointtilde} that the joint distribution of the statistics $\widetilde{inv}$ and $fmaj$ is symmetric over $G_{c,n}$ in the sense of (\ref{jointsym}).

For $n \geq 1$, let $D_{n}^{(c)}=\{\pi \in G_{c,n}: \pi(i)\neq i~\textrm{for}~\textrm{all}~i \in [n]\}$ be the set of all colored derangements in $G_{c,n}$. Specifically, we investigate the Mahonian polynomials defined on the set $D_{n}^{(c)}$  of all colored derangements in the group $G_{c,n}$:  
\begin{equation}\label{cfmaj}
   d_n^{(c)}(q):= \sum_{\pi \in D_{n}^{(c)}}q^{fmaj(\pi)}.
\end{equation}

The use of the length function $L$ together with $fmaj$, which we have been unable to find in the existing literature, leads to some new formulas for the signed Mahonian polynomials on the set $D_n^{(c)}$ of colored derangements in the colored permutations group $G_{c,n}$. In addition, Bagno \cite{bagno2004} introduced a length function $\ell$ on $G_{c,n}$. Signed Mahonian and Euler-Mahonian polynomials studied in the literature are mainly based on the length function $\ell$ (see \cite{chang2021, eu2021}). We remark that the length function $L$, which will be used throughout the paper, is completely different from $\ell$. Utilizing equations (\ref{eflag}) and (\ref{cfmaj}) together, we will produce an explicit formula for the signed Mahonian polynomials over $D_{n}^{(c)}$
\begin{equation}\label{scfmaj}
     \bar{d}_n^{(c)}(q):=\sum_{\pi \in D_{n}^{(c)}}(-1)^{L(\pi)}q^{fmaj(\pi)}
\end{equation}
when $c$ is even.
Gessel \cite{gessel1993} found a remarkable formula for the Mahonian polynomial over the symmetric group $\mathfrak{S}_n$, which is also known as $G_{1,n}$, by major index. Subsequently, another elegant proof was given by Wachs \cite{wachs1989} in a combinatorial way. Chow \cite{chow2006} extended Wachs' method to produce a formula for the Mahonian polynomial over the group of signed permutations $B_n$ (a Weyl group of type $B$), which can be viewed as the group $G_{2,n}$, by employing the flag major index.

In a recent work, Ji and Zhang \cite{ji2025} derived a formula for the Mahonian polynomial of derangements in the even-signed permutations group (a Weyl group of type $D$) by considering the D-major index introduced by Biagioli and Caselli in \cite{biagioli2004},  extending Wachs' approach and utilizing a refinement of Stanley's shuffle theorem produced by themselves in \cite{ji2024}. Moreover, Ji and Zhang \cite{ji2025} established an explicit formula for signed Mahonian polynomial on the set of derangements in each classical Weyl group $W$. 

This paper organized as follows. In the next section, we will provide many signed Mahonian polynomials on $G_{c,n}$ according to the $fmaj$ index together with the length function $L$ and a Mahonian statistic $\widetilde{inv}$. Moreover, we will prove that $\widetilde{inv}$ and $fmaj_{\mathcal{F}}$ have a symmetric joint distribution over $G_{c,n}$. In section 3, we will give a signed Mahonian polynomial on $D_{n}^{(c)}$ for $c$ even. Furthermore, we produce a concrete formula for the difference $d_{n}^{+(c)}-d_{n}^{-(c)}$ between the number of even and odd colored derangements in $G_{c,n}$ for every $c \geq 1$ and $n \geq 1$. We conclude with the paper by leaving an open question.

\section{The signed Mahonian polynomial over $G_{c,n}$}

In this section, our primary focus will be on deriving a formula for the sum
\begin{equation}\label{signedmahoeven}
\sum_{\pi \in G_{c,n}} (-1)^{L(\pi)} q^{fmaj(\pi)} 
\end{equation} 
where $L$ stands for the length function described by the formula (\ref{length}).
We will prove equation (\ref{signedmahoeven}) depending on the parity of $c$. To this end, we will use the formula
\begin{equation}\label{bc2012}
\sum_{\pi \in G_{c,n}} (-1)^{L(|\pi|)} q^{fmaj(\pi)}=[c]_q[2c]_{-q}[3c]_{q}[4c]_{-q}\cdots [nc]_{(-1)^{n-1}q}
\end{equation}
proven by Biagioli and Caselli in \cite{biagioli2012}. 
%by using a decomposition of $G_{c,n}$ in Proposition 4.1 of \cite{biagioli2011}.
%Equation (\ref{bc2012}) can be expressed as
%\begin{equation}\label{bce2012}
%\sum_{\pi \in G_{c,n}} (-1)^{L(|\pi|)} q^{fmaj(\pi)}= \left( \frac{1-q}{1+q}\right)^{\lfloor{\frac{n}{2}}\rfloor} \prod_{i=1}^n [ci]_q.
%\end{equation}

\begin{thm} \label{has}
   If $c$ is an even number, then we have
\begin{equation}\label{hasan}
\sum_{\pi \in G_{c,n}} (-1)^{L(\pi)} q^{fmaj(\pi)}=[c]_{-q}[2c]_q[3c]_{-q}[4c]_{q}\cdots [nc]_{(-1)^{n}q}.
\end{equation}
\end{thm}

\begin{proof} 
Since $c$ is even and considering the length formula given in equation (\ref{length}), we deduce that $ L(\pi) \equiv \textrm{col}(\pi)+inv(|\pi|)~(\textrm{mod}~2)$ for any $\pi \in G_{c,n}$. From equations (\ref{fmaj2001}) and (\ref{bc2012}), we obtain 
\begin{align*}
\sum_{\pi \in G_{c,n}} (-1)^{L(\pi)} q^{fmaj(\pi)}&=  \sum_{\pi \in G_{c,n}} (-1)^{\textrm{col}(\pi)+inv(|\pi|)} q^{fmaj(\pi)}\\
&= \sum_{\pi \in G_{c,n}} (-1)^{L(|\pi|)} (-q)^{fmaj(\pi)}\\
&=[c]_{-q}[2c]_q[3c]_{-q}[4c]_{q}\cdots [nc]_{(-1)^{n}q}
\end{align*}
as desired.
%Assume that $U={\tau \in G_{c,n}:\tau(1)<\tau(2)<\cdots<\tau(n)}$. Biagioli and Zeng \cite{biagioli2011} demonstrated that the group $G_{c,n}$ can be written as $G_{c,n}=U \mathfrak{S}_n$ with $U \cap \mathfrak{S}_n=\{1\}$, that is, for any $\pi \in G_{c,n}$ there exist unique $\tau \in U$ and $\sigma \in \mathfrak{S}_n$ such that $\pi=\tau \sigma$. It is clear that $\textrm{col}(\pi)=\textrm{col}(\tau)$.
%Thus, we can express the formula (\ref{length}) as
%$$ L(\pi) = \textrm{col}(\pi)+c \cdot \sum_{t_j \neq 0} |\{(i,j) : i<j ~\textrm{and}~ \pi_i<\pi_j\}|+inv(|\pi|)$$ 
%for any $\pi \in G_{c,n}$.
 
\end{proof}

Equation (\ref{hasan}) can also be expressed as
\begin{equation}\label{bce2012}
\sum_{\pi \in G_{c,n}} (-1)^{L(\pi)} q^{fmaj(\pi)}= \left( \frac{1-q}{1+q}\right)^{\lfloor{\frac{n+1}{2}}\rfloor} \prod_{i=1}^n [ci]_q
\end{equation}
when $c$ is even. In Table \ref{2}, one can see the distribution of the length function $L$ and flag major index $fmaj$ on $G_{4,2}$. We will illustrate the formula given in (\ref{hasan}) in the following example.

\begin{ex}
Taking Table \ref{2} into account, it can be checked that the formula (\ref{hasan}) holds for $G_{4,2}$. Namely, we have
$$\sum_{\pi \in G_{4,2}} (-1)^{L(\pi)} q^{fmaj(\pi)}=1+q^2-q^8-q^{10}=[4]_{-q}[8]_q.$$
\end{ex}

When we set $c=2$, Theorem \ref{has} corresponds to 
$$\sum_{\pi \in B_n} (-1)^{l_B(\pi)} q^{fmaj(\pi)}=[2]_{-q}[4]_q\cdots [2n]_{(-1)^{n}q}$$
which is a result of Adin, Gessel, and Roichman for the hyperoctahedral group $B_n$ (see Theorem 5.1 in \cite{adin2005}). In \cite{adin2001}, Adin and Roichman proved that the flag major index in the case $c=2$ (i.e., for the hyperoctahedral group $B_n$) is identical with the flag major index in the group $B_n$. It should be also noted here that in the case of $c = 2$ the length function $L$ agrees with the canonical length function $l_B$ of the Weyl group $B_n$ (see \cite{bremke1997}).

\begin{center}
\captionof{table}{Distribution of the length, flag major index and $\widetilde{inv}$ on $G_{4,2}$}\label{2}
\vspace{0.1cm}
\begin{tabularx}{1,045 \textwidth} {
| >{\centering\arraybackslash}X
| >{\centering\arraybackslash}X
| >{\centering\arraybackslash}X
| >{\centering\arraybackslash}X
|| >{\centering\arraybackslash}X
| >{\centering\arraybackslash}X
| >{\centering\arraybackslash}X
| >{\centering\arraybackslash}X
| >{\centering\arraybackslash}X  
| >{\centering\arraybackslash}X
| >{\centering\arraybackslash}X
| >{\centering\arraybackslash}X | }

\hline
$\pi=\pi_1 \pi_2 $&$L(\pi)$&$fmaj(\pi)$&$\widetilde{inv}(\pi)$&$\pi=\pi_1 \pi_2 $&$L(\pi)$&$fmaj(\pi)$&$\widetilde{inv}(\pi)$\\
\hline \hline
1~~~~~~2  & 0  &0&  0 & 2~~~~~~~~$1^{[3]}$&4 &7&7\\
\hline
$1^{[1]}~~~~2 $ & 1  &1&1&$2^{[1]} ~~~~ 1^{[3]}$ &5&8& 8\\
\hline
$ 1^{[2]}~~~~2$  &2 &2&2&$2^{[2]} ~~~~ 1^{[3]}$ &6&9&9\\
\hline
$1^{[3]}~~~~2$  &3 &3&3 &$ 2^{[3]} ~~~~ 1^{[3]}$ &7&10&10\\
\hline
2~~~~~~1  &1 &4&4&$1~~~~~~~~ 2^{[1]}$  &5 &5&1\\
\hline
$ 2^{[1]}~~~~1$  &2 &1&5&$ 1^{[1]}~~~~ 2^{[1]}$  &6&2&2\\
\hline
$ 2^{[2]}~~~~1$  &3 &2&6&$ 1^{[2]}~~~~ 2^{[1]}$ &7&3&3\\
\hline
$2^{[3]}~~~~1$  &4  &3&7&$ 1^{[3]}~~~~ 2^{[1]}$ & 8&4&4\\
\hline
$2~~~~~~~~1^{[1]} $  & 2  &5&5&$1 ~~~~~~~~ 2^{[2]}$ &6&6&2\\
\hline
$ 2^{[1]}~~~~ 1^{[1]}$ & 3  &6&6&$1^{[1]} ~~~~ 2^{[2]}$ &7&7&3\\
\hline
$ 2^{[2]} ~~~~ 1^{[1]}$ &4 &3&7&$ 1^{[2]} ~~~~2^{[2]}$ &8&4&4\\
\hline
$2^{[3]} ~~~~ 1^{[1]}$ &5  &4&8&$ 1^{[3]} ~~~~ 2^{[2]}$ &9&5&5\\
\hline
$2~~~~~~~~ 1^{[2]}$ &3&6&6&$1~~~~~~~~ 2^{[3]}$  &7  &7&3\\
\hline
$ 2^{[1]} ~~~~1^{[2]}$ &4&7&7&$ 1^{[1]}~~~~ 2^{[3]} $ &8  &8&4\\
\hline
$ 2^{[2]} ~~~~1^{[2]}$ &5&8&8&$ 1^{[2]}~~~~2^{[3]}$  &9 &9 &5\\
\hline
$ 2^{[3]} ~~~~1^{[2]}$ &6&5&9&$ 1^{[3]}~~~~2^{[3]}$  &10 &6&6\\
\hline \hline
\end{tabularx}
\end{center}

When $c$ is an odd number, a brute-force computation for $c \leq 7$ and $n \leq 3$ shows that the polynomial 
$$\sum_{\pi \in G_{c,n}} (-1)^{L(\pi)} q^{fmaj(\pi)}$$ 
does not have a nice factorial-type product formula as in (\ref{hasan}).
Similarly to Theorem \ref{has}, we state the following result without proof, based on equations (\ref{fmaj2001}), (\ref{eflag}) and (\ref{tileinv}).

\begin{cor}
When $c$ is an even number, we have
\begin{equation*}
\sum_{\pi \in G_{c,n}} (-1)^{\widetilde{inv}(\pi)} q^{fmaj(\pi)}=\sum_{\pi \in G_{c,n}} (-1)^{fmaj(\pi)} q^{\widetilde{inv}(\pi)}=\prod_{i=1}^n [ci]_{{-q}}. 
\end{equation*}
\end{cor}

\begin{center}
\captionof{table}{Distribution of the length, flag major index and $\widetilde{inv}$ on $G_{3,2}$}\label{g32}
\vspace{0.1cm}

\begin{tabularx}{1,040 \textwidth} {
| >{\centering\arraybackslash}X
| >{\centering\arraybackslash}X
| >{\centering\arraybackslash}X
|>{\centering\arraybackslash}X  
|| >{\centering\arraybackslash}X
| >{\centering\arraybackslash}X
|>{\centering\arraybackslash}X
| >{\centering\arraybackslash}X
| >{\centering\arraybackslash}X
| >{\centering\arraybackslash}X
| >{\centering\arraybackslash}X | }

\hline
$\pi=\pi_1 \pi_2 $&$L(\pi)$&$fmaj(\pi)$&$\widetilde{inv}(\pi)$&$\pi=\pi_1 \pi_2 $&$L(\pi)$&$fmaj(\pi)$&$\widetilde{inv}(\pi)$\\
\hline \hline
$2~~~~~~1$  & 1  &3&3&1~~~~~~$2$ &0&0&0\\
\hline
$2^{[1]}~~~1 $ & 2  &1&4&$1^{[1]} ~~~~ 2$ &1&1&1\\
\hline
$ 2^{[2]}~~~1$  &3 &2&5&$1^{[2]} ~~~ 2$ &2&2&2\\
\hline
$2~~~~~~~~1^{[1]}$  &2 &4&4&$ 1~~~~~~~~ 2^{[1]}$ &4&4&1\\
\hline
$2^{[1]}~~~~1^{[1]}$  &3 &5&5&$1^{[1]}~~~~ 2^{[1]}$  &5 &2&2\\
\hline
$ 2^{[2]}~~~~1^{[1]}$  &4 &3&6&$ 1^{[2]}~~~~ 2^{[1]}$  &6&3&3\\
\hline
$ 2~~~~~~~~1^{[2]}$  &3 &5&5&$1~~~~~~~~ 2^{[2]}$ &5&5&2\\
\hline
$2^{[1]}~~~~1^{[2]}$  &4  &6&6&$ 1^{[1]}~~~~ 2^{[2]}$ & 6&6&3\\
\hline
$2^{[2]}~~~~1^{[2]} $  & 5 &7&7&$1^{[2]} ~~~~ 2^{[2]}$ &7&4&4\\
\hline \hline
\end{tabularx}
\end{center}

%\begin{tabularx}{1 \textwidth} { 
%  | >{\centering\arraybackslash}X 
%  | >{\centering\arraybackslash}X 
%  | >{\centering\arraybackslash}X 
%  || >{\centering\arraybackslash}X 
%  | >{\centering\arraybackslash}X 
%  | >{\centering\arraybackslash}X 
%  | >{\centering\arraybackslash}X 
%  | >{\centering\arraybackslash}X 
%  | >{\centering\arraybackslash}X | }
%\hline
%$\pi=\pi_1 \pi_2 $&$L(\pi)$&$fmaj(\pi)$&$ \pi=\pi_1 \pi_2$  &$L(\pi)$&$fmaj(\pi)$\\
% \hline \hline
%$2~~~~~~1$  & 1  &3&1~~~~~~$2$ &0&0\\
%\hline
%$2^{[1]}~~~~1 $ & 2  &1&$1^{[1]} ~~~~ 2$ &1&1\\
%\hline
%$ 2^{[2]}~~~~1$  &3 &2&$1^{[2]} ~~~~ 2$ &2&2\\
%\hline
%$~~2~~~~~~~1^{[1]}$  &2 &4&$ ~~1~~~~~~ 2^{[1]}$ &4&4\\
%\hline
%$~~2^{[1]}~~~~~~1^{[1]}$  &3 &5&$~~1^{[1]}~~~~~~ 2^{[1]}$  &5 &2\\
%\hline
%$ ~~2^{[2]}~~~~~~1^{[1]}$  &4 &3&$ ~~1^{[2]}~~~~~~ 2^{[1]}$  &6&3\\
%\hline
%$~~ 2~~~~~~1^{[2]}$  &3 &5&$~~1~~~~~~ 2^{[2]}$ &5&5\\
%\hline
%$~~2^{[1]}~~~~1^{[2]}$  &4  &6&$ ~~1^{[1]}~~~~ 2^{[2]}$ & 6&6\\
%\hline
%$~~2^{[2]}~~~~~~1^{[2]} $  & 5 &7&$~~1^{[2]} ~~~~~~ 2^{[2]}$ &7&4\\
%\hline \hline
%\end{tabularx}
%\end{center}

%We verified that the relation
%\begin{equation*}
%\sum_{\pi \in G_{c,n}} (-1)^{\widetilde{inv}(\pi)} q^{fmaj(\pi)}=\prod_{i=1}^n [ci]_{{-q}}. 
%\end{equation*}
%holds for every groups $G_{3,2}$, $G_{5,2}$, $G_{7,2}$ and $G_{3,3}$. This yields to deduce the following conjecture.
\vspace{0.3cm}
The next result immediately follows from the equations (13), (14) and (15) which occur in the proof of Theorem 4.4 in \cite{biagioli2012}.

\begin{cor}\label{conj1}
When $c$ is an odd number, we have
\begin{align*}
 &\sum_{\pi \in G_{c,n}} (-1)^{\widetilde{inv}(\pi)} q^{fmaj(\pi)}=\\
     &\begin{cases}
       [c]_{q}[2c]_{-q}[3c]_{q}[4c]_{-q}\cdots [nc]_{(-1)^{n-1}q},&\quad\text{if $n$ is even}\\
      [c]_{-q}[c]_{q}[2c]_{-q}[3c]_{q}[4c]_{-q}\cdots [(n-2)c]_{q}[(n-1)c]_{-q}[n]_{q^c},&\quad\text{if $n$ is odd.}\\
     \end{cases}
\end{align*}
\end{cor}

Now, we will give the following example to illustrate Corollary \ref{conj1} based on Table \ref{g32}.

\begin{ex}
Considered Table \ref{g32}, we get
$$\sum_{\pi \in G_{3,2}} (-1)^{\widetilde{inv}(\pi)} q^{fmaj(\pi)}=1+q^2-q^3+q^4-q^5-q^7=[3]_q[6]_{-q}.$$
\end{ex}

If we substitute $c \rightarrow 1$ in Corollary \ref{conj1}, then we encounter the polynomial
\begin{equation}\label{gessel-simon}
 \sum_{\pi \in \mathfrak{S}_n} (-1)^{inv(\pi)} q^{maj(\pi)}=[1]_q[2]_{-q}\cdots [n]_{(-1)^{n-1}q} 
\end{equation}
which is nothing else but the result constructed by Gessel and Simon for the symmetric group $\mathfrak{S}_n$ (see Corollary 2 in \cite{wachs1992}). We also realized that  for some $c$ odd numbers the polynomial $\sum_{\pi \in G_{c,n}} (-1)^{fmaj(\pi)} q^{\widetilde{inv}(\pi)}$ has the same formula as in Corollary \ref{conj1}. For instance, this situation may be checked for each group $G_{3,2}, ~G_{5,2},~G_{7,2}$ and $G_{3,3}$. In addition, we observed that the joint distribution of the permutation statistics $\widetilde{inv}$ and $fmaj$ on $G_{c,n}$ is symmetric for some values of $c$ and $n$. Therefore, the above observation allows us to pose the following conjecture, which is computationally confirmed for every group $G_{c,n}$ with $1 \leq c \leq 7$ and $n=2$ as well as for the groups $G_{3,3}$ and $G_{4,3}$ (see Appendix A).

%{\color{red}We observed that the joint distribution of the permutation statistics $\widetilde{inv}$ and $fmaj$ on $G_{c,n}$ is symmetric for some values of $c$ and $n$. Therefore, the above observation allows us to deduce the following conclusion/conjecture. (allows us to conjecture the following formula./allows us to pose the following conjecture, which was computationally verified for $1 \leq c \leq 7$ and $n=2,3$.)}

\begin{conj}\label{symjointtilde}
The statistic $\widetilde{inv}$ has a symmetric joint distribution with $fmaj$ over $G_{c,n}$. That is, 
\begin{equation*}
  \sum_{\pi \in G_{c,n}}t^{\widetilde{inv}(\pi)}q^{fmaj(\pi)}=\sum_{\pi \in G_{c,n}}q^{\widetilde{inv}(\pi)}t^{fmaj(\pi)}.  
\end{equation*}
\end{conj}

It is important to note here that Conjecture \ref{symjointtilde} may be considered as a generalization of the equality (\ref{jointsym}). The best way to prove that the statistics $\widetilde{inv}$ and $fmaj$ on $G_{c,n}$ have a symmetric joint distribution is to find, if possible, an explicit bijection $\psi : G_{c,n} \rightarrow G_{c,n}$ such that for all $ \pi \in G_{c,n}$ 
$$\widetilde{inv}(\pi)=fmaj(\psi(\pi))~~\textrm{and}~~fmaj(\pi)=\widetilde{inv}(\psi(\pi)).$$
In other words, one needs to show that $\psi$ interchanges the two statistics $\widetilde{inv}$ and $fmaj$.

Now we consider the \textit{friends ordering} $\mathcal{F}$ in \cite{fire2005}, where the colors have no effect on many statistics on $G_{c,n}$:
\begin{equation*}
\mathcal{F}:~~1^{[c-1]}<\cdots<1<\cdots<2^{[c-1]}<\cdots<2<\cdots<n^{[c-1]}<\cdots < n.
\end{equation*}
Thus, we notice that $fmaj_{\mathcal{F}}(\pi)=c\cdot maj(\pi)+col(\pi)=c\cdot maj(|\pi|)+col(\pi)$ and $\widetilde{inv}_{\mathcal{F}}(\pi)=\widetilde{inv}(\pi)$ for any $\pi \in G_{c,n}$. It is also known from \cite{fire2005} that $fmaj_{\mathcal{F}}(\pi)$ and $\widetilde{inv}(\pi)$ are equidistributed over $G_{c,n}$. It is not hard to see that 
$fmaj_{\mathcal{F}}$ coincides with $\widetilde{inv}$ on $G_{c,2}$.

\begin{prop}\label{jdfriends}
The statistic $\widetilde{inv}$ has a symmetric joint distribution with $fmaj_{\mathcal{F}}$ over $G_{c,n}$. That is, 
\begin{equation*}
  \sum_{\pi \in G_{c,n}}t^{\widetilde{inv}(\pi)}q^{fmaj_{\mathcal{F}}(\pi)}=\sum_{\pi \in G_{c,n}}q^{\widetilde{inv}(\pi)}t^{fmaj_{\mathcal{F}}(\pi)}.  
\end{equation*}
\end{prop}
\begin{proof}
Assume that $U=\{\tau \in G_{c,n}:\tau(1)<\tau(2)<\cdots<\tau(n)\}$, where the ordering is based on $\mathcal{F}$. Due to Biagioli and Zeng \cite{biagioli2011}, the group $G_{c,n}$ can be written as $G_{c,n}=U \mathfrak{S}_n$ with $U \cap \mathfrak{S}_n=\{1\}$, that is, for any $\pi=\pi_1^{[t_1]} ~\pi_2^{[t_2]} \cdots  \pi_n^{[t_n]} \in G_{c,n}$ with $0\leq t_i \leq c-1$ for all $i=1,\ldots,n$ there exist unique $\tau=1^{[t_1]} ~2^{[t_2]}\cdots  n^{[t_n]} \in U$ such that $\pi=\tau |\pi|$. It is clear that $\textrm{col}(\pi)=\textrm{col}(\tau)$ and $\textrm{maj}(\pi)=\textrm{maj}(|\pi|)$ according to total ordering $\mathcal{F}$. We conclude the desired result from equation (\ref{jointsym}) that
\begin{align*}
  \sum_{\pi \in G_{c,n}}t^{\widetilde{inv}(\pi)}q^{fmaj_{\mathcal{F}}(\pi)}&=\left(\sum_{\tau \in U}t^{col(\tau)}q^{col(\tau)}\right)\left(\sum_{|\pi| \in \mathfrak{S}_n}t^{c \cdot inv(|\pi|)}q^{c \cdot maj(|\pi|)}\right)  \\
  &=\left(\sum_{\tau \in U}t^{col(\tau)}q^{col(\tau)}\right)\left(\sum_{|\pi| \in \mathfrak{S}_n}t^{c \cdot maj(|\pi|)}q^{c \cdot inv(|\pi|)}\right)  \\ 
  &=\sum_{\pi \in G_{c,n}}t^{fmaj_{\mathcal{F}}(\pi)}q^{\widetilde{inv}(\pi)}.
\end{align*}

\end{proof}

\begin{prop}\label{sgnmhninvtilde}
We have 
\begin{align*}
  \sum_{\pi \in G_{c,n}}(-1)^{\widetilde{inv}(\pi)}q^{fmaj_{\mathcal{F}}(\pi)}=
     &\begin{cases}
       \prod_{i=1}^n [ic]_{-q},&\quad\text{if $c$ is even}\\
      \left([c]_{-q}\right)^n\prod_{i=1}^{n} [i]_{(-1)^{i-1} q^c} &\quad\text{if $c$ is odd.}
     \end{cases}
\end{align*}.  
\end{prop}

\begin{proof}
The first case is clearly visible when $c$ is even. Suppose that c is odd. Then we can write from equation (\ref{gessel-simon})
\begin{align*}
  \sum_{\pi \in G_{c,n}}(-1)^{\widetilde{inv}(\pi)}q^{fmaj_{\mathcal{F}}(\pi)}&=\left(\sum_{\tau \in U}(-q)^{col(\tau)}\right)\left(\sum_{|\pi| \in \mathfrak{S}_n}sign(|\pi|)q^{c \cdot maj(|\pi|)}\right)  \\
  &=\left(\sum_{\tau \in U}(-q)^{col(\tau)}\right)[1]_{q^c}[2]_{-{q^c}}\cdots [n]_{(-1)^{n-1}{q^c}}   \\ 
  &=\left(\sum_{c_1+\cdots +c_n=0}^{c-1}(-q)^{c_1+\cdots +c_n}\right)[1]_{q^c}[2]_{-{q^c}}\cdots [n]_{(-1)^{n-1}{q^c}}\\
  &=\left( \sum_{k=0}^{c-1}(-q)^k\right)^n [1]_{q^c}[2]_{-{q^c}}\cdots [n]_{(-1)^{n-1}{q^c}}\\
   &=\left([c]_{-q}\right)^n \prod_{i=1}^{n} [i]_{(-1)^{i-1} q^c}.
\end{align*}
\end{proof}

 As a consequence of Proposition \ref{sgnmhninvtilde}, we can immediately deduce that
\begin{equation*}
\sum_{\pi \in G_{c,n}} (-1)^{L(\pi)} q^{fmaj_{\mathcal{F}}(\pi)}=\left([c]_{-q}\right)^n\prod_{i=1}^{n} [i]_{(-1)^{i-1} q^c}.
\end{equation*}   
when $c$ is even number.

\section{The signed Mahonian polynomial over $D_{n}^{(c)}$}

In this section, we will first give the structure of signed Mahonian polynomial $\bar{d}_{n}^{(c)}(q)=\sum_{\alpha \in D_{n}^{(c)}}(-1)^{L(\alpha)}q^{fmaj(\alpha)}$ over $D_{n}^{(c)}$ when $c$ is even. Subsequently, we will derive the difference between the number of even and odd colored derangements in $G_{c,n}$ for $c$ even.

Following Wachs' work \cite{wachs1989}, for any colored permutation $\pi \in G_{c,A}$, where $A=\{0<a_1<a_2<\cdots < a_k\}$, the \textit{reduction} of $\pi$ to the non-fixed points is defined to be the colored permutation $dp(\pi) \in G_{c,k}$ by replacing each letter $\pi_j^{[r_j]}$ by $i^{[r_j]}$ if $\pi_j=a_i,~i=1,2,\dots,k$. Thus, the \textit{derangement part} $dp(\pi)$ of a colored permutation $\pi \in G_{c,n}$ is the reduction of subword of non-fixed points of $\pi$. For example, $dp(2^{[3]} ~1^{[1]} ~3~4^{[2]}~5)=2^{[3]} ~1^{[1]} ~3^{[2]}$. It should be noted here that the derangement part of a colored permutation is a colored derangement, and conversely, any derangement in $D_{n}^{(c)}$ and a subset having $n-k$ elements of $[n]$ give rise to a unique colored permutation in $G_{c,n}$ with $n-k$ fixed points. 

Now let $\pi=\pi_1^{[t_1]} ~\pi_2^{[t_2]} \cdots  \pi_n^{[t_n]} \in G_{c,n}$. A letter $\pi_i^{[t_i]}$ is called a \textit{subcedant} (respectively an \textit{excedant}) of $\pi$ if $\pi_i<i$ (respectively, $\pi_i>i$) with respect to the total ordering in (\ref{totalord}). It is obvious that all excedants of $\pi$ are positive. Let $s(\pi)$ and $e(\pi)$ denote the number of subcedants and excedants of $\pi$, respectively. Let $\tilde{\pi}$ be a colored permutation having $k$  letters obtained from $\pi \in G_{c,k}$ with $k \leq n$ in the following order:
\begin{enumerate}
    \item [$\bullet$] Replace $i$th smallest (in absolute value) subcedant  $\pi_j^{[t_j]}$ of $\pi$ by $i^{[t_j]}$,~$i=1,2,\ldots,s(\pi)$.
    \item [$\bullet$] Assign $i$th smallest fixed point of $\pi$ to $s(\pi)+1$,~~$i=1,2,\ldots,k-s(\pi)-e(\pi)$.
    \item [$\bullet$] Write $n-i+1$ instead of $i$th largest excedant of $\pi$, where $i$ ranges from $1$ to $e(\pi)$.
\end{enumerate}
Note that if $k=n$, then $\tilde{\pi} \in G_{c,n}$. Moreover, if $\pi \in D_{n}^{(c)}$ then $\tilde{\pi} \in D_{A}^{(c)}$, where $A=\{1,2,\ldots,s(\pi)\}\cup \{n-e(\pi)+1,n-e(\pi)+2,\ldots,n\}$. For instance, the collections of subcedants, excedants and fixed points of $\pi=2 ~1^{[1]} ~3~4^{[2]}~5~7~6^{[2]}$ are, respectively, $\{1^{[1]}, 4^{[2]}, 6^{[2]}\}$, $\{2, 7\}$ and $\{3,5\}$. Since the above construction, $\pi$ corresponds to $\tilde{\pi}=6~3^{[1]}~4~2^{[2]}~5~7~1^{[2]}$. One can observe that descent set of $\pi$ and $\tilde{\pi}$ are the same. This is not an accident.

\begin{lem}\label{firstlem}
Let $\pi \in G_{c,k}$ with $k \leq n$. Then $Des_c(\pi)=Des_c(\tilde{\pi})$ and $col(\pi)=col(\tilde{\pi})$.
\end{lem}

\begin{proof}
It is easy to see from the construction of $\tilde{\pi}$ that the colors remain the same. Based on the total ordering in (\ref{totalord}), we achieved $Des_c(\pi)=Des_c(\tilde{\pi})$ by following the same arguments as in the proof of Lemma 1 in \cite{wachs1989}.
\end{proof}

Given two disjoint permutations $\tau= \tau_1\cdots \tau_j$ and $\sigma= \sigma_1\cdots \sigma_k$, a \textit{shuffle} of $\tau$ and $\sigma$ is a permutation $w=w_1\cdots w_{j+k}$ such that $\tau$ and $\sigma$ appear as subsequences in $w$. The collection of all shuffles of $\tau$ and $\sigma$ is denoted by $sh(\tau,\sigma)$. Garsia and Gessel \cite{garsia1979} showed that the weight of the major index on $sh(\tau,\sigma)$ is
\begin{equation}\label{garsiagessel}
\sum_{w \in sh(\tau,\sigma)}q^{maj(w)}=q^{maj(\tau)+maj(\sigma)} \genfrac{[}{]}{0pt}{}{j+k}{k}_q.  
\end{equation}
Equation (\ref{garsiagessel}) remains true if $\tau$ and $\sigma$ are colored words of different letters from the alphabet in equation (\ref{totalord}). It is apparent that if $\pi \in sh(\tau,\sigma)$ we have $col(\pi)=col(\tau)+col(\sigma)$. Considering the definition of flag major index in (\ref{fmaj2001}), it is not hard to see that the following identity is true:
\begin{equation}\label{cgarsiagessel}
\sum_{\pi \in sh(\tau,\sigma)}q^{fmaj(\pi)}=q^{fmaj(\tau)+fmaj(\sigma)} {j+k \brack k}_{q^c}.  
\end{equation}

The above construction of $\tilde{\pi}$ from $\pi$ yields to following bijection, proved by Assaf in \cite{assaf2010}, is a  generalized version of Theorem 2 in \cite{wachs1989} for the $G_{c,n}$. In the proof of the following lemma, Assaf used the total ordering
\begin{equation*}
n^{[c-1]}<\cdots<1^{[c-1]}<\cdots<n^{[2]}<\cdots<1^{[2]}<n^{[1]}<\cdots<1^{[1]}<1<\cdots < n
\end{equation*}
instead of the total ordering in equation (\ref{totalord}). However, we note that Lemma \ref{secondlem} still remains true when used the total ordering in equation (\ref{totalord}).

\begin{lem} [Assaf \cite{assaf2010}]\label{secondlem}
Let $\alpha \in D_{k}^{(c)}$ with $k \leq n$ and $\gamma=(s(\alpha)+1)\dots (s(\alpha)+k)$. Then the map $\phi: \{\pi \in G_{c,n}: dp(\pi)=\alpha\} \rightarrow sh(\tilde{\alpha},\gamma)$ defined to be $\phi(\pi)=\tilde{\pi}$ is a bijection such that $Des_c(\phi(\pi))=Des_c(\pi)$ and $col(\phi(\pi))=col(\pi)$.
\end{lem}
Thus, the bijection $\phi$ given in Lemma \ref{secondlem} preserves the flag major index, i.e., $fmaj(\phi(\pi))=fmaj(\pi)$ for all $\pi \in \{\pi \in G_{c,n}: dp(\pi)=\alpha\}$. In the following corollary, we will extend Wachs' formula 
$$\mathop{\sum_{dp(\pi)=\sigma}}_{\pi \in \mathfrak{S}_{n}}q^{maj(\pi)}=q^{maj(\sigma)} {n \brack k}_{q}$$
given in \cite{wachs1989} to colored case by employing Lemma \ref{firstlem}, Lemma \ref{secondlem} and equation (\ref{cgarsiagessel}).

\begin{cor}\label{cwac}
Let $0 \leq k \leq n$ and $\sigma \in D_{k}^{(c)}$. Then we have 
\begin{equation*}
\mathop{\sum_{dp(\pi)=\sigma}}_{\pi \in G_{c,n}}q^{fmaj(\pi)}=q^{fmaj(\sigma)} {n \brack k}_{q^c}
\end{equation*}
where ${n \brack k}_{q^c}=\frac{[n]_{q^c}!}{[k]_{q^c}![n-k]_{q^c}!}$ is the $q$-binomial coefficient.
\end{cor}

\begin{prop}\label{14}
Let $0 \leq k \leq n$ and $\sigma \in D_{k}^{(c)}$. If $c$ is even, then we have 
\begin{equation*}
\mathop{\sum_{dp(\pi)=\sigma}}_{\pi \in G_{c,n}}(-1)^{L(\pi)}q^{fmaj(\pi)}=(-1)^{L(\sigma)}q^{fmaj(\sigma)} {n \brack k}_{q^c}.
\end{equation*}
\end{prop}

\begin{proof}
Viewed through the lens of Corollary \ref{cwac}, it suffices to prove that
\begin{equation*}
L(\pi) \equiv L(dp(\pi)) ~(\textrm{mod}~2)
\end{equation*}
for each colored permutation $\pi \in G_{c,n}$. Suppose that the fixed points of $\pi=\pi_1^{[t_1]} ~\pi_2^{[t_2]} \cdots  \pi_n^{[t_n]}$ are in increasing order $p_1 < p_2 <\dots < p_k$. Since $L(\pi)=inv_c(\pi)$, the contribution of inversions at exactly each fixed point $p_j$ to the total number of inversions $inv_c(\pi)$ can be obtained from the proof of Theorem 4.1 in \cite{ahmia2025} as $2(p_j-j)(n-k-p_j+j)$ and so the contribution of this case to the inversion number $inv_c(\pi)$ is even. On the other hand, any $\pi_r^{[t_r]}$ with $t_r \neq 0$ appearing after the position of $p_j$ and greater than $p_j$ contributes the number $c$ to $inv_c(\pi)$.
Thus, removing all fixed points from $\pi$, $inv_c(\pi)$ decreases by an even number since $c$ is even. In conclusion, we deduce that $L(\pi) \equiv L(dp(\pi)) ~(\textrm{mod}~2)$.
\end{proof}

\begin{thm}\label{t5}
For any $n \geq 1$, we have
\begin{equation}\label{derpoly}
d_{n}^{(c)}(q)=\sum_{\pi \in D_{n}^{(c)}}q^{fmaj(\pi)}=[c]_q[2c]_q\dots [nc]_q \sum_{k=0}^n \frac{(-1)^k q^{c \binom{k}{2}}}{[c]_q[2c]_q\dots [kc]_q}.
\end{equation}
\end{thm}

\begin{proof}
Considering Corollary \ref{cwac} and equation (\ref{eflag}), we have
%and the fact that $G_{c,n}=\coprod_{k=0}^n D_k^{(c)}$, 
\begin{align*}
[c]_q[2c]_q\dots [nc]_q&=\sum_{\pi \in G_{c,n}}q^{fmaj(\pi)}\\
&=\sum_{k=0}^n~\sum_{\alpha \in D_k^{(c)}}~\sum_{dp(\pi)=\alpha}q^{fmaj(\pi)}\\
&=\sum_{k=0}^n~\sum_{\alpha \in D_k^{(c)}}q^{fmaj(\alpha)}{n \brack k}_{q^c}\\
&=\sum_{k=0}^n {n \brack k}_{q^c} d_k^{(c)}(q).
\end{align*}
Hence, Gauss inversion (i.e., $q$-binomial inversion) given in Corollary 3.38 of \cite{aigner1979} on the resulting equation leads to the desired formula (\ref{derpoly}).
\end{proof}

For any $n \geq 1$, we denote by $d_{n}^{(c)}$ the number of all colored derangements in $D_n^{(c)}$. Then 
\begin{equation}\label{assafdr}
    d_{n}^{(c)}=n! \sum_{k=0}^{n} \frac{(-1)^k c^{n-k}}{k!}.
\end{equation}

It should be emphasized that Assaf \cite{assaf2010} established the formula for $d_{n}^{(c)}$ in (\ref{assafdr}) by using Principle of Inclusion and Exclusion. Next result of this paper is the following theorem, giving us the formula for the signed Mahonian polynomial over $D_n^{(c)}$.

\begin{thm}\label{t6}
Let $c$ be even and $n \geq 1$. Then, we have
$$\bar{d}_{n}^{(c)}(q)=\sum_{\alpha \in D_{n}^{(c)}}(-1)^{L(\alpha)}q^{fmaj(\alpha)}=[c]_q[2c]_q\dots [nc]_q \sum_{k=0}^n \frac{(-1)^k q^{c \binom{k}{2}}}{[c]_q[2c]_q\dots [kc]_q} \left( \frac{1-q}{1+q}\right)^{\lfloor{\frac{n-k+1}{2}}\rfloor}.$$
\end{thm}

\begin{proof}
Summing over all colored derangements $\alpha \in D_k^{(c)}$ with $0 \leq k \leq n$, we can deduce from Proposition \ref{14} and equation (\ref{scfmaj}) that
\begin{align*}
 \sum_{\pi \in G_{c,n}}(-1)^{L(\pi)}q^{fmaj(\pi)}&=\sum_{k=0}^n~\sum_{\alpha \in D_k^{(c)}}~\sum_{dp(\pi)=\alpha,~\pi \in G_{c,n}}(-1)^{L(\pi)}q^{fmaj(\pi)}\\
&=\sum_{k=0}^n~\sum_{\alpha \in D_k^{(c)}}~(-1)^{L(\alpha)}q^{fmaj(\alpha)} {n \brack k}_{q^c} \\
 &=\sum_{k=0}^n {n \brack k}_{q^c}\bar{d}_{k}^{(c)}.
\end{align*}
Taking the formula in (\ref{bce2012}) into account, we get
\begin{equation}\label{gauss16}
    \sum_{k=0}^n {n \brack k}_{q^c}\bar{d}_{k}^{(c)}=\left( \frac{1-q}{1+q}\right)^{\lfloor{\frac{n+1}{2}}\rfloor} \prod_{i=1}^n [ci]_q.
\end{equation}
It is also well known from the proof of Theorem 2.2 in \cite{ahmia2025} that $[ci]_q=[c]_q[i]_{q^c}$ for any positive integer $i$. Applying Gauss inversion to equation (\ref{gauss16}) yields to 
\begin{align*}
   \bar{d}_{n}^{(c)}(q)&=\sum_{k=0}^n {n \brack k}_{q^c}(-1)^{n-k}q^{c \binom{n-k}{2} }\left( \frac{1-q}{1+q}\right)^{\lfloor{\frac{k+1}{2}}\rfloor} \prod_{i=1}^k [ci]_q\\
   &=[c]_q[2c]_q\dots [nc]_q \sum_{k=0}^n \frac{(-1)^k q^{c \binom{k}{2}}}{[c]_q[2c]_q\dots [kc]_q}\left( \frac{1-q}{1+q}\right)^{\lfloor{\frac{n-k+1}{2}}\rfloor}
\end{align*}
which is equivalent to the desired formula. 
\end{proof}

Let $D_n^{+(c)}$ denote the sets of colored derangements with even length (according to the length function $L$) in $G_{c,n}$. Define $d_n^{+(c)}:=|D_n^{+(c)}|$. We infer from Theorem \ref{t5} and Theorem \ref{t6} that 
\begin{equation}\label{even17}
   d_n^{+(c)}(q)=[c]_q[2c]_q\dots [nc]_q \sum_{k=0}^n \frac{(-1)^k q^{c \binom{k}{2}}}{[c]_q[2c]_q\dots [kc]_q} \left( \frac{1}{2}+\frac{1}{2}\left( \frac{1-q}{1+q}\right)^{\lfloor{\frac{n-k+1}{2}}\rfloor}\right) 
\end{equation}
when $c$ is even. In similar vein, we indicate by $D_n^{-(c)}$ the set of odd colored derangements in $G_{c,n}$. Let $d_n^{-(c)}:=|D_n^{-(c)}|$. Substituting $q \rightarrow 1$ in equation (\ref{even17}), we obtain 
$$d_n^{+(c)}=\frac{n!}{2} \sum_{k=0}^{n-1} \frac{(-1)^k c^{n-k}}{k!} +(-1)^n.$$
Moreover, we conclude that $d_n^{-(c)}=\frac{n!}{2} \sum_{k=0}^{n-1} \frac{(-1)^k c^{n-k}}{k!}$. Thus, the difference between the number of even and odd colored derangements in $G_{c,n}$ is 
\begin{equation}\label{diff}
  d_n^{+(c)}-d_n^{-(c)}=(-1)^n  
\end{equation}
when $c$ is even. As a matter of the fact, the formula (\ref{diff}) is actually a general form of the difference, which can be easily obtained from \cite{ji2025}, for the hyperoctahedral group corresponding to the case $c=2$. Since the number $d_n^{+(2)}$ of even length derangements in the hyperoctahedral group $B_n$ has been already computed by Ji and Zhang in \cite{ji2025}, one can also count the number of odd derangements in the group $B_n$. Due to \cite{ji2025}, it can easily be seen that the difference in question is equal to zero for the even-signed permutation group.

%Based on Theorem \ref{t5} and Theorem \ref{t6}, we have
%\begin{equation}\label{odd18}
%   d_n^{odd(c)}(q9=[c]_q[2c]_q\dots [nc]_q \sum_{k=0}^n \frac{(-1)^k q^{c \binom{k}{2}}}{[c]_q[2c]_q\dots [kc]_q} \left( \frac{1}{2}-\frac{1}{2}\left( \frac{1-q}{1+q}\right)^{\lfloor{\frac{n-k+1}{2}}\rfloor}\right). 
%\end{equation}
\vspace{0.2cm}

\begin{thm} \label{drgmntsgn}
Let $c$ and $n$ be positive integers. Then, we have 
$$\sum_{\pi \in D_n^{(c)}} \operatorname{sgn}(\pi)=(-1)^{n-1}(cn-1),$$
where $\operatorname{sgn}(\pi):=(-1)^{inv(|\pi|)}$.
 \end{thm}

\begin{proof}
Write $\pi = (\sigma; c_1,\dots,c_n) \in G_{c,n}$, where $\sigma$ is the underlying permutation and $c_1,\ldots,c_n \in \{0,\ldots,c-1\}$ . For a subset $S \subseteq [n]$, define  
\[
F_S = \{\pi \in G_{c,n} \mid \sigma(i) = i \text{ and } c_i = 0 \text{ for all } i \in S \}.
\]  
If $|S| = k$, the elements of $F_S$ are obtained by choosing a permutation $\sigma'$ of $T = [n] \setminus S$ (with $|T| = n-k$) and colors $c_t \in \{0, \dots, c-1\}$ for $t \in T$. Hence $|F_S| = (n-k)! \, c^{\,n-k}$.

Now examine the sign-weighted sum  
\[
f(F_S) = \sum_{\pi \in F_S} \operatorname{sgn}(\pi).
\]  
Since $\sigma$ is the identity on $S$, we have $\operatorname{sgn}(\pi) = \operatorname{sgn}(\sigma')$, where $\sigma'$ is the permutation on $T$. Therefore  
\[
f(F_S) = \Bigl( \sum_{\sigma' \in \mathfrak{S}_T} \operatorname{sgn}(\sigma') \Bigr) \cdot c^{\,n-k}.
\]  
where $\mathfrak{S}_T$ denotes the set of all permutations of $T$.
The sum $\sum_{\sigma' \in \mathfrak{S}_m} \operatorname{sgn}(\sigma')$ is $0$ for $m \ge 2$ and equals $1$ for $m = 0, 1$. Consequently, $f(F_S) \neq 0$ only when $n - k \le 1$, i.e., $k \ge n-1$.

The non-zero contributions are:  

\begin{itemize}
    \item When $k = n$ (so $T = \varnothing$), the unique permutation of $T$ is the identity (with sign $1$), giving $f(F_S) = 1$.
    \item When $k = n-1$ (so $|T| = 1$), the unique permutation of $T$ also has sign $1$, and there are $c$ choices for the color of the single element in $T$; thus $f(F_S) = c$.
\end{itemize}

By inclusion--exclusion,  
\[
\sum_{\pi \in D_n^{(c)}} \operatorname{sgn}(\pi)
= \sum_{k=0}^n (-1)^k \sum_{|S| = k} f(F_S).
\]  
Only $k = n$ and $k = n-1$ contribute:  
\[
\sum_{|S| = n} f(F_S) = 1, \qquad 
\sum_{|S| = n-1} f(F_S) = \binom{n}{n-1} \, c = n c.
\]  
Hence,  
\[
\sum_{\pi \in D_n^{(c)}} \operatorname{sgn}(\pi)= (-1)^n \cdot 1 + (-1)^{n-1} \cdot n c
= (-1)^n \bigl( 1 - n c \bigr),
\]  
which completes the argument.
\end{proof}

In the case of symmetric group $\mathfrak{S}_n$, Chapman \cite{chapman2001} provided a bijective proof showing that the difference between the number of even and odd derangements in $\mathfrak{S}_n$ is 
\begin{equation}\label{chapman}
  d_n^{+}-d_n^{-}=(-1)^{n-1}(n-1).  
\end{equation}

Substituting $c=1$ into Theorem \ref{drgmntsgn}, we obtain Chapman's result above. To illustrate Theorem \ref{drgmntsgn}, we will give the following example.

\begin{ex}
When using Table \ref{g33}, we observe that $\sum_{\pi \in D_3^{(3)}} \operatorname{sgn}(\pi)=8.$
\end{ex}

%\begin{lem}\label{lem:parityL}
%Let $\pi=(\sigma;c_1,\dots,c_n)\in G_{c,n}$. Then
%\[
%(-1)^{L(\pi)} 
%=
%\begin{cases}
%(-1)^{\operatorname{inv}(\sigma)+\operatorname{col}(\pi)+\sum_{\substack{c_j \neq 0\\ i<j, ~~ \sigma_i<\sigma_j}} 1}, & \text{if $c$ is odd},\\[4pt]
%(-1)^{\operatorname{inv}(\sigma)+\operatorname{col}(\pi)}, & \text{if $c$ is even}.
%\end{cases}
%\]
%\end{lem}

\begin{thm}\label{thm:signed-strong-derangements}
The difference between the number of even and odd colored derangements (according to the length function $L$) in $G_{c,n}$ is:
\[
d_{n}^{+(c)}-d_{n}^{-(c)}=\sum_{\pi\in D_n^{(c)}} (-1)^{L(\pi)}
=
\begin{cases}
(-1)^{n}, & \text{if $c$ is even},\\[4pt]
(-1)^{n-1}(n-1), & \text{if $c$ is odd}.
\end{cases}
\]
\end{thm}

\begin{proof}
Let $\pi = (\sigma;c_1,\dots,c_n)$ be the colored permutation
\[\pi=\sigma_1^{[c_1]}\cdots \sigma_n^{[c_n]},\]
where $\sigma\in \mathfrak{S}_n$ and $c_j\in\{0,1,\dots,c-1\}$.  

\medskip
\begin{enumerate}
\item \textbf{The case $c$ even.}
Denote
\[
\operatorname{Fix}(\sigma)=\{i\in[n]\mid \sigma(i)=i\},
\qquad
k=|\operatorname{Fix}(\sigma)|.
\]

If $i\in\operatorname{Fix}(\sigma)$, the condition $\pi(i)\neq i^{[0]}$ forces $c_i\in\{1,\dots,c-1\}$;  
if $i\notin\operatorname{Fix}(\sigma)$, then $c_i\in\{0,\dots,c-1\}$.  
Consequently,
\begin{equation}\label{eq:color-sum}
\sum_{\substack{(c_1,\dots,c_n)\\ (\sigma;\mathbf{c})\in D_n^{(c)}}}
(-1)^{\sum_{i=1}^n c_i}
=
\Bigl(\sum_{a=0}^{c-1}(-1)^a\Bigr)^{\,n-k}
\Bigl(\sum_{a=1}^{c-1}(-1)^a\Bigr)^{\,k}. 
\end{equation}

From the definition of $L$, we have
\[
(-1)^{L(\pi)}=(-1)^{\operatorname{inv}(\sigma)}(-1)^{\sum_{i=1}^n c_i}.
\]

Hence
\[
\sum_{\pi\in D_n^{(c)}}(-1)^{L(\pi)}
=
\sum_{\sigma\in\mathfrak{S}_n}(-1)^{\operatorname{inv}(\sigma)}
\sum_{\substack{(c_1,\dots,c_n)\\ (\sigma;\mathbf{c})\in D_n^{(c)}}}
(-1)^{\sum_{i=1}^n c_i}.
\]

For even $c$,
\[
\sum_{a=0}^{c-1}(-1)^a=0,\qquad
\sum_{a=1}^{c-1}(-1)^a=-1.
\]

Thus the right-hand side of \eqref{eq:color-sum} vanishes unless $k=n$; i.e. $\sigma=\operatorname{id}$.  
For the identity permutation the inner sum equals $(-1)^n$, whence
\[
\sum_{\pi\in D_n^{(c)}}(-1)^{L(\pi)}
=(-1)^{\operatorname{inv}(\operatorname{id})}(-1)^n=(-1)^n.
\]

\item \textbf{The case $c$ odd.}

For $j\in[n]$ define
\[
N(\sigma,j):=\bigl|\{i<j\mid \sigma(i)<\sigma(j)\}\bigr|,
\]
the number of non-inversions of $\sigma$ to the left of position $j$.  
Then, from the definition of $L$, we obtain
\[
(-1)^{L(\pi)}=(-1)^{\operatorname{inv}(\sigma)
+\sum_{j:\,c_j\neq 0}N(\sigma,j)
+\sum_{j=1}^{n}c_j},
\]
and consequently
\begin{equation}\label{eq:L-factor}
(-1)^{L(\pi)}
=
(-1)^{\operatorname{inv}(\sigma)}
\prod_{j=1}^{n}(-1)^{c_j}
\prod_{j:\,c_j\neq 0}(-1)^{N(\sigma,j)}.
\end{equation}

Fix $\sigma\in \mathfrak{S}_n$ and sum \eqref{eq:L-factor} over all colorings
$(c_1,\dots,c_n)\in\{0,1,\dots,c-1\}^n$ satisfying the derangement condition
$c_j=0\Rightarrow\sigma(j)\neq j$.  
Define
\[
S(\sigma):=\sum_{\substack{c_1,\dots,c_n\\ c_j=0\Rightarrow\sigma(j)\neq j}}
(-1)^{L(\pi)}.
\]

Factor out $(-1)^{\operatorname{inv}(\sigma)}$ and, for each $j$, set $p_j=(-1)^{N(\sigma,j)}$.

\begin{itemize}
\item \textit{If $\sigma(j)=j$ (a fixed point)}, then $c_j\neq0$ and $c_j\in\{1,\dots,c-1\}$.  
The contribution from such a position is
\[
\sum_{c_j=1}^{c-1}(-1)^{c_j} p_j
=p_j\sum_{k=1}^{c-1}(-1)^{k}.
\]
Because $c$ is odd, $c-1$ is even, and $\sum_{k=1}^{c-1}(-1)^k=0$.  
Hence if $\sigma$ has \textbf{any} fixed point, $S(\sigma)=0$.

\item \textit{If $\sigma(j)\neq j$ (non-fixed point)}, then $c_j\in\{0,\dots,c-1\}$.  
The contribution is
\[
\sum_{c_j=0}^{c-1}
\Bigl(\mathbf{1}_{c_j=0}
+\mathbf{1}_{c_j\neq0}\,(-1)^{c_j}p_j\Bigr)
=1+p_j\sum_{k=1}^{c-1}(-1)^k=1.
\]
\end{itemize}

Thus only classical \textbf{derangements} $\sigma$ contribute, and for them
\[
S(\sigma)=(-1)^{\operatorname{inv}(\sigma)}.
\]

\medskip

Summing over all permutations yields
\[
\sum_{\pi\in D_n^{(c)}}(-1)^{L(\pi)}
=\sum_{\sigma\in\mathfrak{D}_n}(-1)^{\operatorname{inv}(\sigma)}=(-1)^{\,n-1}(n-1),
\]
where $\mathfrak{D}_n$ denotes the set of ordinary derangements in $\mathfrak{S}_n$.
\end{enumerate}

Combining the two cases completes the proof.
\end{proof}

In conclusion, we can state that Theorem \ref{thm:signed-strong-derangements} generalizes Chapman's result in equation (\ref{chapman}) to the colored permutation groups.\\

\textbf{Open Question:} 
\begin{itemize}
    \item [$\bullet$] \textit{It is a natural question to ask here what the structure of the signed Mahonian polynomial $\bar{d}_{n}^{(c)}(q)$ is exactly when $c$ is odd?}
\end{itemize}

\section{Appendix A}
 In this section, we will present Table \ref{g33} and Table \ref{g43}, which show the distributions of $fmaj$ and $\widetilde{inv}$ on the groups $G_{3,3}$ and $G_{4,3}$, respectively. These tables help one to satisfy that $fmaj$ and $\widetilde{inv}$ have a symmetric joint distribution on both $G_{3,3}$ and $G_{4,3}$.

\begin{table}[t]

	\caption{Distribution of $L$, $fmaj$, $\widetilde{inv}$ and $fmaj_{\mathcal{F}}$ on $G_{3,3}$}\label{g33}

\label{tab:table1}

\scalebox{0.67}{

\begin{tabular}{|m{0.6cm}|m{0.3cm}m{0.3cm}m{0.3cm}|m{0.5cm}|m{0.8cm}|m{0.6cm}|m{0.9cm}||m{0.6cm}|m{0.3cm}m{0.3cm}m{0.3cm}|m{0.5cm}|m{0.8cm}|m{0.6cm}|m{0.9cm}||m{0.6cm}|m{0.3cm}m{0.3cm}m{0.3cm}|m{0.5cm}|m{0.8cm}|m{0.6cm}|m{0.9cm}|}

\hline

\centering

Rank&\textbf{$\pi_1~\pi_2~\pi_3$}&&&L&$fmaj$&$\widetilde{inv}$&$fmaj_{\mathcal{F}}$&Rank&\textbf{$\pi_1~~\pi_2~~~\pi_3$}&&&L&$fmaj$&$\widetilde{inv}$&$fmaj_{\mathcal{F}}$&Rank&\textbf{$\pi_1~~\pi_2~~~ \pi_3$}&&&L&$fmaj$&$\widetilde{inv}$&$fmaj_{\mathcal{F}}$\\

\hline\hline

$\mathbf{1}$&1 &2 & 3 &0&0 &0&0 &$\mathbf{55}$&1 &2 & $3^{[1]}$ &7&7&1&1 &$\mathbf{109}$& 1 &2 & $3^{[2]}$ &8&8 &2 &2\\

\hline

$\mathbf{2}$&$1^{[1]}$& 2& 3 &1 &1 & 1&1&$\mathbf{56}$& $1^{[1]}$ &2 & $3^{[1]}$&8&8 &2 &2 &$\mathbf{110}$& $1^{[1]}$ &2 & $3^{[2]}$ &9& 9& 3&3 \\

\hline

$\mathbf{3}$&$1^{[2]}$& 2& 3 &2 &2  &2&2&$\mathbf{57}$& $1^{[2]}$& 2& $3^{[1]}$  &9&9 &3 &3&$\mathbf{111}$& $1^{[2]}$ &2 & $3^{[2]}$&10 &10 &4&4\\

\hline

$\mathbf{4}$&1& $2^{[1]}$& 3 &4&4 &1&1 &$\mathbf{58}$& 1& $2^{[1]}$& $3^{[1]}$ &11&5 &2&2 &$\mathbf{112}$&  1& $2^{[1]}$& $3^{[2]}$ &12&12 &3&3\\

\hline

$\mathbf{5}$&$1^{[1]}$& $2^{[1]}$& 3  &5& 2 &2&2&$\mathbf{59}$& $1^{[1]}$& $2^{[1]}$& $3^{[1]}$&12 &3 & 3&3&$\mathbf{113}$&   $1^{[1]}$& $2^{[1]}$& $3^{[2]}$ &13&10 &4&4\\

\hline

$\mathbf{6}$&$1^{[2]}$& $2^{[1]}$& 3  &6&  3&3&3&$\mathbf{60}$&  $1^{[2]}$& $2^{[1]}$& $3^{[1]}$ &13& 4& 4&4&$\mathbf{114}$&   $1^{[2]}$& $2^{[1]}$& $3^{[2]}$ &14&11 &5&5\\

\hline

$\mathbf{7}$&1& $2^{[2]}$& 3 &5&5 &2&2 &$\mathbf{61}$& 1& $2^{[2]}$& $3^{[1]}$ &12& 6&3 &3&$\mathbf{115}$& 1& $2^{[2]}$& $3^{[2]}$&13& 7&4&4 \\

\hline

$\mathbf{8}$&$1^{[1]}$& $2^{[2]}$& 3 &6 & 6 &3&3&$\mathbf{62}$& $1^{[1]}$& $2^{[2]}$& $3^{[1]}$ &13& 7 &4&4 &$\mathbf{116}$& $1^{[1]}$& $2^{[2]}$& $3^{[2]}$&14&8 &5&5 \\

\hline

$\mathbf{9}$&$1^{[2]}$& $2^{[2]}$& 3 &7 &4 &4&4 &$\mathbf{63}$& $1^{[2]}$& $2^{[2]}$& $3^{[1]}$&14& 5 &5 &5&$\mathbf{117}$& $1^{[2]}$& $2^{[2]}$& $3^{[2]}$&15&6 &6&6 \\

\hline

$\mathbf{10}$&2& 1& 3 &1&3 & 3&3&$\mathbf{64}$& 2& 1& $3^{[1]}$ &8&10&4&4&$\mathbf{118}$&  2& 1& $3^{[2]}$&9&11 &5&5\\
\hline

$\mathbf{11}$&$2^{[1]}$& 1& 3 &2& 1&4&4&$\mathbf{65}$&$2^{[1]}$& 1& $3^{[1]}$&9&8 &5&5&$\mathbf{119}$&  $2^{[1]}$& 1& $3^{[2]}$&10&9 & 6&6\\
\hline

$\mathbf{12}$&$2^{[2]}$& 1& 3 &3& 2& 5&5&$\mathbf{66}$& $2^{[2]}$& 1& $3^{[1]}$ &10&9&6&6&$\mathbf{120}$&  $2^{[2]}$& 1& $3^{[2]}$&11&10 &7&7 \\

\hline

$\mathbf{13}$&2& $1^{[1]}$& 3 &2&4 &4&4 &$\mathbf{67}$&2& $1^{[1]}$& $3^{[1]}$ &9&5 &5&5&$\mathbf{121}$&  2& $1^{[1]}$& $3^{[2]}$&10&12 &6&6 \\

\hline

$\mathbf{14}$&$2^{[1]}$& $1^{[1]}$& 3 &3&5 & 5&5&$\mathbf{68}$& $2^{[1]}$& $1^{[1]}$& $3^{[1]}$ &10&6 &6&6&$\mathbf{122}$&  $2^{[1]}$& $1^{[1]}$& $3^{[2]}$&11& 13&7&7\\

\hline

$\mathbf{15}$&$2^{[2]}$& $1^{[1]}$& 3 &4&3 & 6&6&$\mathbf{69}$&$2^{[2]}$& $1^{[1]}$& $3^{[1]}$&11&4 &7&7&$\mathbf{123}$&  $2^{[2]}$& $1^{[1]}$& $3^{[2]}$&12&11 &8&8 \\

\hline

$\mathbf{16}$&2& $1^{[2]}$& 3 &3&5 &5&5&$\mathbf{70}$& 2& $1^{[2]}$& $3^{[1]}$&10&6  &6&6&$\mathbf{124}$&  2& $1^{[2]}$& $3^{[2]}$&11& 7&7&7 \\

\hline

$\mathbf{17}$&$2^{[1]}$& $1^{[2]}$& 3 &4&6&6&6 &$\mathbf{71}$& $2^{[1]}$& $1^{[2]}$& $3^{[1]}$ &11&7&7&7&$\mathbf{125}$&  $2^{[1]}$& $1^{[2]}$& $3^{[2]}$&12&8 & 8&8\\

\hline

$\mathbf{18}$&$2^{[2]}$& $1^{[2]}$& 3 &5& 7& 7&7&$\mathbf{72}$& $2^{[2]}$& $1^{[2]}$& $3^{[1]}$&12&8 &8&8&$\mathbf{126}$&  $2^{[2]}$& $1^{[2]}$& $3^{[2]}$&13& 9&9&9 \\
\hline

\hline
$\mathbf{19}$&1& 3& 2  &1& 6& 3&6&$\mathbf{73}$&1& 3& $2^{[1]}$ &5 &7  &4&7&$\mathbf{127}$& 1& 3& $2^{[2]}$&6& 8& 5&8\\

\hline

$\mathbf{20}$&$1^{[1]}$ & 3& 2  &2&7 & 4&7&$\mathbf{74}$&$1^{[1]}$ & 3& $2^{[1]}$  &6&8  &5&8&$\mathbf{128}$& $1^{[1]}$ & 3& $2^{[2]}$&7&9&6&9 \\

\hline

$\mathbf{21}$&$1^{[2]}$ & 3& 2  &3&8 &5 &8&$\mathbf{75}$&$1^{[2]}$ & 3& $2^{[1]}$  &7&9  &6&9&$\mathbf{129}$& $1^{[2]}$ & 3& $2^{[2]}$&8&10 &7&10\\

\hline

$\mathbf{22}$&1& $3^{[1]}$& 2 &5&4 & 4&7&$\mathbf{76}$&1& $3^{[1]}$& $2^{[1]}$  &9&11  &5&8&$\mathbf{130}$& 1& $3^{[1]}$& $2^{[2]}$&10&12 &6&9 \\
\hline

$\mathbf{23}$& $1^{[1]}$& $3^{[1]}$& 2 &6 & 2&5&8 &$\mathbf{77}$&$1^{[1]}$& $3^{[1]}$& $2^{[1]}$  &10& 9 &6&9&$\mathbf{131}$& $1^{[1]}$& $3^{[1]}$& $2^{[2]}$&11&10 &7&10 \\
\hline

$\mathbf{24}$& $1^{[2]}$& $3^{[1]}$& 2  &7& 3&6&9 &$\mathbf{78}$&$1^{[2]}$& $3^{[1]}$& $2^{[1]}$  &11& 10 &7&10&$\mathbf{132}$& $1^{[2]}$& $3^{[1]}$& $2^{[2]}$&12& 11&8&11 \\

\hline
$\mathbf{25}$&1& $3^{[2]}$& 2 & 6& 5&5&8&$\mathbf{79}$&1& $3^{[2]}$& $2^{[1]}$  &10&6  &6&9&$\mathbf{133}$& 1& $3^{[2]}$& $2^{[2]}$&11&13 &7&10\\
\hline

$\mathbf{26}$& $1^{[1]}$& $3^{[2]}$& 2  &7&6 &6 &9&$\mathbf{80}$&$1^{[1]}$& $3^{[2]}$& $2^{[1]}$  &11& 7 &7&10&$\mathbf{134}$& $1^{[1]}$& $3^{[2]}$& $2^{[2]}$&12&14 &8&11\\
\hline

$\mathbf{27}$& $1^{[2]}$& $3^{[2]}$& 2  &8& 4&7&10 &$\mathbf{81}$&$1^{[2]}$& $3^{[2]}$& $2^{[1]}$  &12& 5 &8&11&$\mathbf{135}$& $1^{[2]}$& $3^{[2]}$& $2^{[2]}$&13& 12&9&12\\

\hline

$\mathbf{28}$&3& 1& 2  &2&3 &6&3 &$\mathbf{82}$&3& 1& $2^{[1]}$  &6&10  &7&4&$\mathbf{136}$& 3& 1& $2^{[2]}$&7&11 &8&5 \\

\hline

$\mathbf{29}$&$3^{[1]}$ & 1& 2  &3&1 & 7&4&$\mathbf{83}$&$3^{[1]}$ & 1& $2^{[1]}$  &7& 8 &8&5&$\mathbf{137}$& $3^{[1]}$ & 1& $2^{[2]}$&8&9 &9&6 \\

\hline

$\mathbf{30}$&$3^{[2]}$ & 1& 2  &4& 2&8&5 &$\mathbf{84}$&$3^{[2]}$ & 1& $2^{[1]}$  &8&9  &9&6&$\mathbf{138}$& $3^{[2]}$ & 1& $2^{[2]}$&9&10 &10&7 \\
\hline

$\mathbf{31}$&3& $1^{[1]}$& 2 &3& 4& 7&4&$\mathbf{85}$&3& $1^{[1]}$& $2^{[1]}$  &7& 5 &8&5&$\mathbf{139}$& 3&$1^{[1]}$& $2^{[2]}$&8&12 & 9&6\\
\hline

$\mathbf{32}$&$3^{[1]}$& $1^{[1]}$& 2 & 4&5& 8&5&$\mathbf{86}$&$3^{[1]}$& $1^{[1]}$& $2^{[1]}$  &8& 6 &9&6&$\mathbf{140}$& $3^{[1]}$&$1^{[1]}$& $2^{[2]}$&9&13 &10&7 \\
\hline

$\mathbf{33}$&$3^{[2]}$& $1^{[1]}$& 2 &5 &3& 9&6&$\mathbf{87}$&$3^{[2]}$& $1^{[1]}$& $2^{[1]}$&9  & 4 &10&7&$\mathbf{141}$& $3^{[2]}$&$1^{[1]}$& $2^{[2]}$&10&11 &11&8 \\

\hline
$\mathbf{34}$&3& $1^{[2]}$& 2 &4&5 &8&5 &$\mathbf{88}$&3& $1^{[2]}$& $2^{[1]}$  &8&6  &9&6&$\mathbf{142}$& 3&$1^{[2]}$& $2^{[2]}$&9&7 &10&7\\
\hline

$\mathbf{35}$& $3^{[1]}$& $1^{[2]}$& 2 &5&6 &9&6 &$\mathbf{89}$&$3^{[1]}$& $1^{[2]}$& $2^{[1]}$  &9& 7 &10&7&$\mathbf{143}$& $3^{[1]}$&$1^{[2]}$& $2^{[2]}$&10& 8& 11&8\\
\hline

$\mathbf{36}$& $3^{[2]}$& $1^{[2]}$& 2 &6&7 & 10&7&$\mathbf{90}$&$3^{[2]}$& $1^{[2]}$& $2^{[1]}$  &10&8  &11&8&$\mathbf{144}$& $3^{[2]}$&$1^{[2]}$& $2^{[2]}$&11&9 &12&9 \\

\hline
$\mathbf{37}$&2&3 & 1&2 &6 &6&6 &$\mathbf{91}$&2&3 & $1^{[1]}$&3&7&7 &7&$\mathbf{145}$&2&3 & $1^{[2]}$ &4&8&8&8 \\

\hline

$\mathbf{38}$&$2^{[1]}$&3 & 1 &3&7 &7&7 &$\mathbf{92}$&$2^{[1]}$&3 & $1^{[1]}$&4&8&8&8 &$\mathbf{146}$&$2^{[1]}$&3 & $1^{[2]}$ &5&9 &9&9\\

\hline

$\mathbf{39}$&$2^{[2]}$&3 & 1 &4&8 &8&8 &$\mathbf{93}$&$2^{[2]}$&3 & $1^{[1]}$&5&9&9&9 &$\mathbf{147}$&$2^{[2]}$&3 & $1^{[2]}$ &6& 10&10&10\\

\hline

$\mathbf{40}$&2&$3^{[1]}$ & 1 &6&4 &7&7 &$\mathbf{94}$&2&$3^{[1]}$ & $1^{[1]}$&7&11&8&8 &$\mathbf{148}$&2&$3^{[1]}$ & $1^{[2]}$ &8& 12&9&9\\

\hline

$\mathbf{41}$&$2^{[1]}$&$3^{[1]}$ & 1 &7&2 & 8&8&$\mathbf{95}$&$2^{[1]}$&$3^{[1]}$ & $1^{[1]}$&8&9& 9&9&$\mathbf{149}$&$2^{[1]}$&$3^{[1]}$ & $1^{[2]}$&9 & 10&10&10\\

\hline

$\mathbf{42}$&$2^{[2]}$&$3^{[1]}$ & 1 &8&3 & 9&9&$\mathbf{96}$&$2^{[2]}$&$3^{[1]}$ & $1^{[1]}$&9&10&10 &10&$\mathbf{150}$&$2^{[2]}$&$3^{[1]}$ & $1^{[2]}$ &10& 11&11&11\\

\hline

$\mathbf{43}$&2&$3^{[2]}$ & 1 &7 &5&8&8 &$\mathbf{97}$&2&$3^{[2]}$ & $1^{[1]}$&8&6&9&9 &$\mathbf{151}$&2&$3^{[2]}$ & $1^{[2]}$ &9& 13&10&10\\

\hline

$\mathbf{44}$&$2^{[1]}$&$3^{[2]}$ & 1 & 8&6&9&9 &$\mathbf{98}$&$2^{[1]}$&$3^{[2]}$ & $1^{[1]}$&9&7& 10&10&$\mathbf{152}$&$2^{[1]}$&$3^{[2]}$ & $1^{[2]}$ &10&14 &11&11\\
\hline

$\mathbf{45}$&$2^{[2]}$&$3^{[2]}$ & 1&9 &4 &10&10 &$\mathbf{99}$&$2^{[2]}$&$3^{[2]}$ & $1^{[1]}$&10&5&11&11 &$\mathbf{153}$&$2^{[2]}$&$3^{[2]}$ & $1^{[2]}$ &11& 12&12&12\\

\hline

$\mathbf{46}$&3&2 & 1 &3&9 &9&9 &$\mathbf{100}$&3&2 & $1^{[1]}$&4&10&10 &10&$\mathbf{154}$&3&2 & $1^{[2]}$ &5&11 &11&11\\

\hline

$\mathbf{47}$&$3^{[1]}$&2 & 1 &4&7 &10 &10&$\mathbf{101}$&$3^{[1]}$&2 & $1^{[1]}$&5&8&11&11 &$\mathbf{155}$&$3^{[1]}$&2 & $1^{[2]}$ &6& 9&12&12\\

\hline

$\mathbf{48}$&$3^{[2]}$&2 & 1 &5&8 &11&11 &$\mathbf{102}$&$3^{[2]}$&2 & $1^{[1]}$&6&9&12&12 &$\mathbf{156}$&$3^{[2]}$&2 & $1^{[2]}$&7 &10 &13&13\\

\hline

$\mathbf{49}$&3&$2^{[1]}$ & 1 &4&4 &10&10 &$\mathbf{103}$&3&$2^{[1]}$ & $1^{[1]}$&5&11&11&11 &$\mathbf{157}$&3&$2^{[1]}$ & $1^{[2]}$ &6& 12&12&12\\

\hline

$\mathbf{50}$&$3^{[1]}$&$2^{[1]}$ & 1 &5&5 &11&11 &$\mathbf{104}$&$3^{[1]}$&$2^{[1]}$ & $1^{[1]}$&6&12&12&12 &$\mathbf{158}$&$3^{[1]}$&$2^{[1]}$ & $1^{[2]}$&7 &13 &13&13\\
\hline

$\mathbf{51}$&$3^{[2]}$&$2^{[1]}$ & 1 &6&3 &12&12 &$\mathbf{105}$&$3^{[2]}$&$2^{[1]}$ & $1^{[1]}$&7&10& 13&13&$\mathbf{159}$&$3^{[2]}$&$2^{[1]}$ & $1^{[2]}$&8 &11 &14&14\\

\hline

$\mathbf{52}$&3&$2^{[2]}$ & 1 &5& 5&11&11 &$\mathbf{106}$&3&$2^{[2]}$ & $1^{[1]}$&6&6&12&12 &$\mathbf{160}$&3&$2^{[2]}$ & $1^{[2]}$ &7&13 &13&13\\

\hline

$\mathbf{53}$&$3^{[1]}$&$2^{[2]}$ & 1 &6& 6& 12&12&$\mathbf{107}$&$3^{[1]}$&$2^{[2]}$ & $1^{[1]}$&7&7&13 &13&$\mathbf{161}$&$3^{[1]}$&$2^{[2]}$ & $1^{[2]}$ &8&14 &14&14\\

\hline

$\mathbf{54}$&$3^{[2]}$&$2^{[2]}$ & 1 &7&7 & 13&13&$\mathbf{108}$&$3^{[2]}$&$2^{[2]}$ & $1^{[1]}$&8&8&14&14 &$\mathbf{162}$&$3^{[2]}$&$2^{[2]}$ & $1^{[2]}$ &9&15 &15&15\\
\hline
\end{tabular}}
\end{table}

\renewcommand{\arraystretch}{0.0 }

\begin{table}[p]

	\caption{Distribution of $fmaj$ and $\widetilde{inv}$ on $G_{4,3}$}\label{g43}

\begin{center}
\scalebox{0.44}{

\begin{tabular}{|m{0.9cm}|m{0.4cm}m{0.4cm}m{0.5cm}|m{1.4cm}|m{0.9cm}||m{0.9cm}|m{0.4cm}m{0.4cm}m{0.5cm}|m{1.4cm}|m{0.9cm}||m{0.9cm}|m{0.4cm}m{0.4cm}m{0.5cm}|m{1.4cm}|m{0.9cm}||m{0.9cm}|m{0.4cm}m{0.4cm}m{0.5cm}|m{1.4cm}|m{1.0cm}|}

\hline

\centering

Rank&\textbf{$\pi_1~\pi_2~\pi_3$}&&&$fmaj(\pi)$&$\widetilde{inv}(\pi)$&Rank&\textbf{$\pi_1~~\pi_2~~~\pi_3$}&&&$fmaj(\pi)$&$\widetilde{inv}(\pi)$&Rank&\textbf{$\pi_1~~\pi_2~~~ \pi_3$}&&&$fmaj(\pi)$&$\widetilde{inv}(\pi)$&Rank&\textbf{$\pi_1~\pi_2~\pi_3$}&&&$fmaj(\pi)$&$\widetilde{inv}(\pi)$\\

\hline\hline

$\mathbf{1}$&1 &2 & 3 & 0& 0&$\mathbf{97}$&1 &2 & $3^{[1]}$ &9&1 &$\mathbf{193}$& 1 &2 & $3^{[2]}$ &10 &2 &$\mathbf{289}$&1 &2 & $3^{[3]}$ &11 &3\\

\hline

$\mathbf{2}$&$1^{[1]}$& 2& 3  &1 & 1&$\mathbf{98}$& $1^{[1]}$ &2 & $3^{[1]}$&10 & 2 &$\mathbf{194}$& $1^{[1]}$ &2 & $3^{[2]}$ &11 &3&$\mathbf{290}$& $1^{[1]}$ &2 & $3^{[3]}$&12 &4  \\

\hline

$\mathbf{3}$&$1^{[2]}$& 2& 3  & 2 &2&$\mathbf{99}$& $1^{[2]}$& 2& $3^{[1]}$  & 11&3 &$\mathbf{195}$& $1^{[2]}$ &2 & $3^{[2]}$ &12 &4&$\mathbf{291}$& $1^{[2]}$ &2 & $3^{[3]}$&13&5\\

\hline

$\mathbf{4}$&$1^{[3]}$& 2& 3  &3 &3&$\mathbf{100}$& $1^{[3]}$& 2& $3^{[1]}$ &12&4 &$\mathbf{196}$& $1^{[3]}$ &2 & $3^{[2]}$ &13 &5&$\mathbf{292}$& $1^{[3]}$ &2 & $3^{[3]}$&14&6\\

\hline

$\mathbf{5}$&1& $2^{[1]}$& 3 & 5& 1&$\mathbf{101}$& 1& $2^{[1]}$& $3^{[1]}$ &6&2 &$\mathbf{197}$&  1& $2^{[1]}$& $3^{[2]}$ & 15&3&$\mathbf{293}$&  1& $2^{[1]}$& $3^{[3]}$&16&4\\

\hline

$\mathbf{6}$&$1^{[1]}$& $2^{[1]}$& 3  & 2 &2&$\mathbf{102}$& $1^{[1]}$& $2^{[1]}$& $3^{[1]}$ &3 & 3&$\mathbf{198}$&   $1^{[1]}$& $2^{[1]}$& $3^{[2]}$ &12 &4&$\mathbf{294}$&  $1^{[1]}$& $2^{[1]}$& $3^{[3]}$&13&5\\

\hline

$\mathbf{7}$&$1^{[2]}$& $2^{[1]}$& 3  & 3 &3&$\mathbf{103}$&  $1^{[2]}$& $2^{[1]}$& $3^{[1]}$ &4 & 4&$\mathbf{199}$&   $1^{[2]}$& $2^{[1]}$& $3^{[2]}$ &13 &5&$\mathbf{295}$&  $1^{[2]}$& $2^{[1]}$& $3^{[3]}$&14&6\\

\hline

$\mathbf{8}$&$1^{[3]}$& $2^{[1]}$& 3  & 4&4 &$\mathbf{104}$& $1^{[3]}$& $2^{[1]}$& $3^{[1]}$ &5&5 &$\mathbf{200}$&  $1^{[3]}$& $2^{[1]}$& $3^{[2]}$ & 14&6&$\mathbf{296}$&  $1^{[3]}$& $2^{[1]}$& $3^{[3]}$&15&7\\
\hline

$\mathbf{9}$&1& $2^{[2]}$& 3 &6 &2 &$\mathbf{105}$& 1& $2^{[2]}$& $3^{[1]}$ &7 &3 &$\mathbf{201}$& 1& $2^{[2]}$& $3^{[2]}$&8 & 4&$\mathbf{297}$& 1& $2^{[2]}$& $3^{[3]}$&17&5\\

\hline

$\mathbf{10}$&$1^{[1]}$& $2^{[2]}$& 3  & 7 &3&$\mathbf{106}$& $1^{[1]}$& $2^{[2]}$& $3^{[1]}$ &  8& 4&$\mathbf{202}$& $1^{[1]}$& $2^{[2]}$& $3^{[2]}$& 9& 5&$\mathbf{298}$& $1^{[1]}$& $2^{[2]}$& $3^{[3]}$&18&6\\

\hline

$\mathbf{11}$&$1^{[2]}$& $2^{[2]}$& 3  &4 &4 &$\mathbf{107}$& $1^{[2]}$& $2^{[2]}$& $3^{[1]}$& 5 & 5&$\mathbf{203}$& $1^{[2]}$& $2^{[2]}$& $3^{[2]}$&6 &6 &$\mathbf{299}$& $1^{[2]}$& $2^{[2]}$& $3^{[3]}$&15&7\\

\hline

$\mathbf{12}$&$1^{[3]}$& $2^{[2]}$& 3  & 5&5 &$\mathbf{108}$& $1^{[3]}$& $2^{[2]}$& $3^{[1]}$ & 6& 6&$\mathbf{204}$& $1^{[3]}$& $2^{[2]}$& $3^{[2]}$&7 &7 &$\mathbf{300}$& $1^{[3]}$& $2^{[2]}$& $3^{[3]}$&16&8\\
\hline

$\mathbf{13}$&1& $2^{[3]}$& 3 &7 &3 &$\mathbf{109}$& 1& $2^{[3]}$& $3^{[1]}$ & 8 &4 &$\mathbf{205}$& 1& $2^{[3]}$& $3^{[2]}$&9 & 5&$\mathbf{301}$& 1& $2^{[3]}$& $3^{[3]}$&10&6\\

\hline

$\mathbf{14}$&$1^{[1]}$& $2^{[3]}$& 3  &8&4 &$\mathbf{110}$& $1^{[1]}$& $2^{[3]}$& $3^{[1]}$ &9  &5 &$\mathbf{206}$& $1^{[1]}$& $2^{[3]}$& $3^{[2]}$& 10& 6&$\mathbf{302}$& $1^{[1]}$& $2^{[3]}$& $3^{[3]}$&11&7\\

\hline

$\mathbf{15}$&$1^{[2]}$& $2^{[3]}$& 3  &9& 5&$\mathbf{111}$& $1^{[2]}$& $2^{[3]}$& $3^{[1]}$& 10&6 &$\mathbf{207}$& $1^{[2]}$& $2^{[3]}$& $3^{[2]}$&11 &7 &$\mathbf{303}$& $1^{[2]}$& $2^{[3]}$& $3^{[3]}$&12&8\\

\hline

$\mathbf{16}$&$1^{[3]}$& $2^{[3]}$& 3 &6& 6&$\mathbf{112}$& $1^{[3]}$& $2^{[3]}$& $3^{[1]}$&7 &7 &$\mathbf{208}$& $1^{[3]}$& $2^{[3]}$& $3^{[2]}$& 8& 8&$\mathbf{304}$& $1^{[3]}$& $2^{[3]}$& $3^{[3]}$&9&9\\

\hline

$\mathbf{17}$&2& 1& 3 &4 &4 &$\mathbf{113}$& 2& 1& $3^{[1]}$ &13&5&$\mathbf{209}$&  2& 1& $3^{[2]}$&14 &6 &$\mathbf{305}$&  2& 1& $3^{[3]}$&15&7\\
\hline

$\mathbf{18}$&$2^{[1]}$& 1& 3 &1 &5&$\mathbf{114}$&$2^{[1]}$& 1& $3^{[1]}$&10 &6&$\mathbf{210}$&  $2^{[1]}$& 1& $3^{[2]}$&11 &7 &$\mathbf{306}$&  $2^{[1]}$& 1& $3^{[3]}$&12&8\\
\hline

$\mathbf{19}$&$2^{[2]}$& 1& 3 &2 & 6&$\mathbf{115}$& $2^{[2]}$& 1& $3^{[1]}$ &11&7&$\mathbf{211}$&  $2^{[2]}$& 1& $3^{[2]}$&12 &8 &$\mathbf{307}$&  $2^{[2]}$& 1& $3^{[3]}$&13&9\\

\hline

$\mathbf{20}$&$2^{[3]}$& 1& 3 &3 &7 &$\mathbf{116}$& $2^{[3]}$& 1& $3^{[1]}$  &12 &8&$\mathbf{212}$&  $2^{[3]}$& 1& $3^{[2]}$&13 &9 &$\mathbf{308}$&  $2^{[3]}$& 1& $3^{[3]}$&14&10\\

\hline

$\mathbf{21}$&2& $1^{[1]}$& 3 &5 &5 &$\mathbf{117}$&2& $1^{[1]}$& $3^{[1]}$ &6 &6&$\mathbf{213}$&  2& $1^{[1]}$& $3^{[2]}$&15 & 7&$\mathbf{309}$&  2& $1^{[1]}$& $3^{[3]}$&16&8\\

\hline

$\mathbf{22}$&$2^{[1]}$& $1^{[1]}$& 3 & 6&6 &$\mathbf{118}$& $2^{[1]}$& $1^{[1]}$& $3^{[1]}$ &7 &7&$\mathbf{214}$&  $2^{[1]}$& $1^{[1]}$& $3^{[2]}$&16 &8 &$\mathbf{310}$&  $2^{[1]}$& $1^{[1]}$& $3^{[3]}$&17&9\\

\hline

$\mathbf{23}$&$2^{[2]}$& $1^{[1]}$& 3 &3 &7 &$\mathbf{119}$&$2^{[2]}$& $1^{[1]}$& $3^{[1]}$&4 &8&$\mathbf{215}$&  $2^{[2]}$& $1^{[1]}$& $3^{[2]}$&13 &9 &$\mathbf{311}$&  $2^{[2]}$& $1^{[1]}$& $3^{[3]}$&14&10\\

\hline

$\mathbf{24}$&$2^{[3]}$& $1^{[1]}$& 3 & 4& 8&$\mathbf{120}$& $2^{[3]}$& $1^{[1]}$& $3^{[1]}$  & 5 &9&$\mathbf{216}$&  $2^{[3]}$& $1^{[1]}$& $3^{[2]}$& 14& 10&$\mathbf{312}$&  $2^{[3]}$& $1^{[1]}$& $3^{[3]}$&15&11\\

\hline

$\mathbf{25}$&2& $1^{[2]}$& 3 &6 &6&$\mathbf{121}$& 2& $1^{[2]}$& $3^{[1]}$&7  &7&$\mathbf{217}$&  2& $1^{[2]}$& $3^{[2]}$&8 &8 &$\mathbf{313}$&  2& $1^{[2]}$& $3^{[3]}$&17&9\\

\hline

$\mathbf{26}$&$2^{[1]}$& $1^{[2]}$& 3 &7&7 &$\mathbf{122}$& $2^{[1]}$& $1^{[2]}$& $3^{[1]}$ &8&8&$\mathbf{218}$&  $2^{[1]}$& $1^{[2]}$& $3^{[2]}$& 9& 9&$\mathbf{314}$&  $2^{[1]}$& $1^{[2]}$& $3^{[3]}$&18&10\\

\hline

$\mathbf{27}$&$2^{[2]}$& $1^{[2]}$& 3 &8 &8 &$\mathbf{123}$& $2^{[2]}$& $1^{[2]}$& $3^{[1]}$& 9&9&$\mathbf{219}$&  $2^{[2]}$& $1^{[2]}$& $3^{[2]}$& 10&10 &$\mathbf{315}$&  $2^{[2]}$& $1^{[2]}$& $3^{[3]}$&19&11\\
\hline

$\mathbf{28}$&$2^{[3]}$& $1^{[2]}$& 3 &5 & 9&$\mathbf{124}$&$2^{[3]}$& $1^{[2]}$ & $3^{[1]}$&6 &10&$\mathbf{220}$&  $2^{[3]}$& $1^{[2]}$& $3^{[2]}$&7 &11&$\mathbf{316}$&  $2^{[3]}$& $1^{[2]}$& $3^{[3]}$&16&12\\

\hline

$\mathbf{29}$&2& $1^{[3]}$& 3 &7 &7 &$\mathbf{125}$&2& $1^{[3]}$& $3^{[1]}$ & 8&8&$\mathbf{221}$&  2& $1^{[3]}$& $3^{[2]}$& 9& 9&$\mathbf{317}$&  2& $1^{[3]}$& $3^{[3]}$&10&10\\

\hline

$\mathbf{30}$&$2^{[1]}$& $1^{[3]}$& 3 &8 & 8&$\mathbf{126}$&$2^{[1]}$& $1^{[3]}$ & $3^{[1]}$&9&9&$\mathbf{222}$&  $2^{[1]}$& $1^{[3]}$& $3^{[2]}$&10 &10 &$\mathbf{318}$&  $2^{[1]}$& $1^{[3]}$& $3^{[3]}$&11&11\\

\hline

$\mathbf{31}$&$2^{[2]}$& $1^{[3]}$& 3 &9 &9 &$\mathbf{127}$& $2^{[2]}$& $1^{[3]}$& $3^{[1]}$&10 &10&$\mathbf{223}$&  $2^{[2]}$& $1^{[3]}$& $3^{[2]}$& 11&11 &$\mathbf{319}$&  $2^{[2]}$& $1^{[3]}$& $3^{[3]}$&12&12\\

\hline

$\mathbf{32}$&$2^{[3]}$& $1^{[3]}$& 3 &10 &10 &$\mathbf{128}$& $2^{[3]}$& $1^{[3]}$& $3^{[1]}$& 11 &11&$\mathbf{224}$&  $2^{[3]}$& $1^{[3]}$& $3^{[2]}$&12 &12&$\mathbf{320}$&  $2^{[3]}$& $1^{[3]}$& $3^{[3]}$&13&13\\

\hline
$\mathbf{33}$&1& 3& 2  & 8&4 &$\mathbf{129}$&1& 3& $2^{[1]}$  &9  &5&$\mathbf{225}$& 1& 3& $2^{[2]}$& 10&6 &$\mathbf{321}$& 1& 3& $2^{[3]}$&11&7\\

\hline

$\mathbf{34}$&$1^{[1]}$ & 3& 2  & 9&5 &$\mathbf{130}$&$1^{[1]}$ & 3& $2^{[1]}$  &10  &6&$\mathbf{226}$& $1^{[1]}$ & 3& $2^{[2]}$&11 & 7&$\mathbf{322}$& $1^{[1]}$ & 3& $2^{[3]}$&12&8\\

\hline

$\mathbf{35}$&$1^{[2]}$ & 3& 2  & 10&6 &$\mathbf{131}$&$1^{[2]}$ & 3& $2^{[1]}$  &11  &7&$\mathbf{227}$& $1^{[2]}$ & 3& $2^{[2]}$&12 &8 &$\mathbf{323}$& $1^{[2]}$ & 3& $2^{[3]}$&13&9\\
\hline

$\mathbf{36}$&$1^{[3]}$ & 3& 2  &11 &7 &$\mathbf{132}$&$1^{[3]}$ & 3& $2^{[1]}$  &  12&8&$\mathbf{228}$& $1^{[3]}$ & 3& $2^{[2]}$&13 &9 &$\mathbf{324}$& $1^{[3]}$ & 3& $2^{[3]}$&14&10\\

\hline

$\mathbf{37}$&1& $3^{[1]}$& 2 &5 &5 &$\mathbf{133}$&1& $3^{[1]}$& $2^{[1]}$  & 14 &6&$\mathbf{229}$& 1& $3^{[1]}$& $2^{[2]}$&15 &7 &$\mathbf{325}$& 1& $3^{[1]}$& $2^{[3]}$&16&8\\
\hline

$\mathbf{38}$& $1^{[1]}$& $3^{[1]}$& 2  &2 &6 &$\mathbf{134}$&$1^{[1]}$& $3^{[1]}$& $2^{[1]}$  &11  &7&$\mathbf{230}$& $1^{[1]}$& $3^{[1]}$& $2^{[2]}$& 12&8 &$\mathbf{326}$& $1^{[1]}$& $3^{[1]}$& $2^{[3]}$&13&9\\
\hline

$\mathbf{39}$& $1^{[2]}$& $3^{[1]}$& 2  &3 &7 &$\mathbf{135}$&$1^{[2]}$& $3^{[1]}$& $2^{[1]}$  & 12 &8&$\mathbf{231}$& $1^{[2]}$& $3^{[1]}$& $2^{[2]}$&13 &9 &$\mathbf{327}$& $1^{[2]}$& $3^{[1]}$& $2^{[3]}$&14&10\\

\hline

$\mathbf{40}$& $1^{[3]}$& $3^{[1]}$& 2  &4 &8 &$\mathbf{136}$&$1^{[3]}$& $3^{[1]}$& $2^{[1]}$  &13  &9&$\mathbf{232}$& $1^{[3]}$& $3^{[1]}$& $2^{[2]}$&14 & 10&$\mathbf{328}$& $1^{[3]}$& $3^{[1]}$& $2^{[3]}$&15&11\\

\hline
$\mathbf{41}$&1& $3^{[2]}$& 2 &6 & 6&$\mathbf{137}$&1& $3^{[2]}$& $2^{[1]}$  &7  &7&$\mathbf{233}$& 1& $3^{[2]}$& $2^{[2]}$& 16&8 &$\mathbf{329}$& 1& $3^{[2]}$& $2^{[3]}$&17&9\\
\hline

$\mathbf{42}$& $1^{[1]}$& $3^{[2]}$& 2  & 7&7 &$\mathbf{138}$&$1^{[1]}$& $3^{[2]}$& $2^{[1]}$  &8  &8&$\mathbf{234}$& $1^{[1]}$& $3^{[2]}$& $2^{[2]}$&17 &9 &$\mathbf{330}$& $1^{[1]}$& $3^{[2]}$& $2^{[3]}$&18&10\\
\hline

$\mathbf{43}$& $1^{[2]}$& $3^{[2]}$& 2  & 4&8 &$\mathbf{139}$&$1^{[2]}$& $3^{[2]}$& $2^{[1]}$  &5  &9&$\mathbf{235}$& $1^{[2]}$& $3^{[2]}$& $2^{[2]}$&14 &10 &$\mathbf{331}$& $1^{[2]}$& $3^{[2]}$& $2^{[3]}$&15&11\\

\hline

$\mathbf{44}$& $1^{[3]}$& $3^{[2]}$& 2  &5 &9&$\mathbf{140}$&$1^{[3]}$& $3^{[2]}$& $2^{[1]}$  & 6 &10&$\mathbf{236}$& $1^{[3]}$& $3^{[2]}$& $2^{[2]}$&15 &11 &$\mathbf{332}$& $1^{[3]}$& $3^{[2]}$& $2^{[3]}$&16&12\\

\hline

$\mathbf{45}$&1& $3^{[3]}$& 2 &7 &7 &$\mathbf{141}$&1& $3^{[3]}$& $2^{[1]}$  &  8&8&$\mathbf{237}$& 1& $3^{[3]}$& $2^{[2]}$&9 &9 &$\mathbf{333}$& 1& $3^{[3]}$& $2^{[3]}$&18&10\\

\hline

$\mathbf{46}$&$1^{[1]}$& $3^{[3]}$& 2 &8 &8 &$\mathbf{142}$&$1^{[1]}$& $3^{[3]}$& $2^{[1]}$  &9  &9&$\mathbf{238}$& $1^{[1]}$& $3^{[3]}$& $2^{[2]}$&10 &10 &$\mathbf{334}$& $1^{[1]}$& $3^{[3]}$& $2^{[3]}$&19&11\\

\hline

$\mathbf{47}$&$1^{[2]}$& $3^{[3]}$& 2 &9 &9 &$\mathbf{143}$&$1^{[2]}$& $3^{[3]}$& $2^{[1]}$  & 10 &10&$\mathbf{239}$& $1^{[2]}$& $3^{[3]}$& $2^{[2]}$& 11&11 &$\mathbf{335}$& $1^{[2]}$& $3^{[3]}$& $2^{[3]}$&20&12\\

\hline

$\mathbf{48}$&$1^{[3]}$& $3^{[3]}$& 2 &6 & 10&$\mathbf{144}$&$1^{[3]}$& $3^{[3]}$& $2^{[1]}$  & 7 &11&$\mathbf{240}$& $1^{[3]}$& $3^{[3]}$& $2^{[2]}$&8 & 12&$\mathbf{336}$& $1^{[3]}$& $3^{[3]}$& $2^{[3]}$&17&13\\
\hline
$\mathbf{49}$&3& 1& 2  &4 &8 &$\mathbf{145}$&3& 1& $2^{[1]}$  &  13&9&$\mathbf{241}$& 3& 1& $2^{[2]}$&14 &10 &$\mathbf{337}$& 3& 1& $2^{[3]}$&15&11\\

\hline

$\mathbf{50}$&$3^{[1]}$ & 1& 2  &1 &9 &$\mathbf{146}$&$3^{[1]}$ & 1& $2^{[1]}$  &  10&10&$\mathbf{242}$& $3^{[1]}$ & 1& $2^{[2]}$&11 & 11&$\mathbf{338}$& $3^{[1]}$ & 1& $2^{[3]}$&12&12\\

\hline

$\mathbf{51}$&$3^{[2]}$ & 1& 2  &2 &10 &$\mathbf{147}$&$3^{[2]}$ & 1& $2^{[1]}$  &  11&11&$\mathbf{243}$& $3^{[2]}$ & 1& $2^{[2]}$&12 &12 &$\mathbf{339}$& $3^{[2]}$ & 1& $2^{[3]}$&13&13\\
\hline

$\mathbf{52}$&$3^{[3]}$ & 1& 2  &3 &11 &$\mathbf{148}$&$3^{[3]}$ & 1& $2^{[1]}$  &  12&12&$\mathbf{244}$& $3^{[3]}$ & 1& $2^{[2]}$&13 &13 &$\mathbf{340}$& $3^{[3]}$ & 1& $2^{[3]}$&14&14\\

\hline

$\mathbf{53}$&3& $1^{[1]}$& 2 & 5&9 &$\mathbf{149}$&3& $1^{[1]}$& $2^{[1]}$  &  6&10&$\mathbf{245}$& 3&$1^{[1]}$& $2^{[2]}$&15 &11 &$\mathbf{341}$& 3& $1^{[1]}$& $2^{[3]}$&16&12\\
\hline

$\mathbf{54}$&$3^{[1]}$& $1^{[1]}$& 2 &6 &10 &$\mathbf{150}$&$3^{[1]}$& $1^{[1]}$& $2^{[1]}$  & 7 &11&$\mathbf{246}$& $3^{[1]}$&$1^{[1]}$& $2^{[2]}$& 16&12 &$\mathbf{342}$& $3^{[1]}$& $1^{[1]}$& $2^{[3]}$&17&13\\
\hline

$\mathbf{55}$&$3^{[2]}$& $1^{[1]}$& 2 &3 &11 &$\mathbf{151}$&$3^{[2]}$& $1^{[1]}$& $2^{[1]}$  &4  &12&$\mathbf{247}$& $3^{[2]}$&$1^{[1]}$& $2^{[2]}$&13 &13 &$\mathbf{343}$& $3^{[2]}$& $1^{[1]}$& $2^{[3]}$&14&14\\

\hline

$\mathbf{56}$&$3^{[3]}$& $1^{[1]}$& 2 &4 &12 &$\mathbf{152}$&$3^{[3]}$& $1^{[1]}$& $2^{[1]}$  & 5 &13&$\mathbf{248}$& $3^{[3]}$&$1^{[1]}$& $2^{[2]}$&14 &14 &$\mathbf{344}$& $3^{[3]}$& $1^{[1]}$& $2^{[3]}$&15&15\\

\hline
$\mathbf{57}$&3& $1^{[2]}$& 2 &6 & 10&$\mathbf{153}$&3& $1^{[2]}$& $2^{[1]}$  &  7&11&$\mathbf{249}$& 3&$1^{[2]}$& $2^{[2]}$&8 &12 &$\mathbf{345}$& 3& $1^{[2]}$& $2^{[3]}$&17&13\\
\hline

$\mathbf{58}$& $3^{[1]}$& $1^{[2]}$& 2 &7 &11 &$\mathbf{154}$&$3^{[1]}$& $1^{[2]}$& $2^{[1]}$  & 8 &12&$\mathbf{250}$& $3^{[1]}$&$1^{[2]}$& $2^{[2]}$&9 &13 &$\mathbf{346}$& $3^{[1]}$& $1^{[2]}$& $2^{[3]}$&18&14\\
\hline

$\mathbf{59}$& $3^{[2]}$& $1^{[2]}$& 2 &8 &12 &$\mathbf{155}$&$3^{[2]}$& $1^{[2]}$& $2^{[1]}$  & 9 &13&$\mathbf{251}$& $3^{[2]}$&$1^{[2]}$& $2^{[2]}$&10 &14 &$\mathbf{347}$& $3^{[2]}$& $1^{[2]}$& $2^{[3]}$&19&15\\
\hline

$\mathbf{60}$&$3^{[3]}$& $1^{[2]}$& 2 &5 &13 &$\mathbf{156}$&$3^{[3]}$& $1^{[2]}$& $2^{[1]}$  & 6 &14&$\mathbf{252}$& $3^{[3]}$&$1^{[2]}$& $2^{[2]}$&7 & 15&$\mathbf{348}$& $3^{[3]}$& $1^{[2]}$& $2^{[3]}$&16&16\\
\hline

$\mathbf{61}$&3& $1^{[3]}$& 2 &7 &11 &$\mathbf{157}$&3& $1^{[3]}$& $2^{[1]}$  &  8&12&$\mathbf{253}$& 3&$1^{[3]}$& $2^{[2]}$& 9&13 &$\mathbf{349}$& 3& $1^{[3]}$& $2^{[3]}$&10&14\\

\hline

$\mathbf{62}$&$3^{[1]}$& $1^{[3]}$& 2 &8 &12 &$\mathbf{158}$&$3^{[1]}$& $1^{[3]}$& $2^{[1]}$  & 9 &13&$\mathbf{254}$& $3^{[1]}$&$1^{[3]}$& $2^{[2]}$&10 &14 &$\mathbf{350}$& $3^{[1]}$& $1^{[3]}$& $2^{[3]}$&11&15\\

\hline

$\mathbf{63}$&$3^{[2]}$& $1^{[3]}$& 2 & 9&13&$\mathbf{159}$&$3^{[2]}$& $1^{[3]}$& $2^{[1]}$  & 10 &14&$\mathbf{255}$& $3^{[2]}$&$1^{[3]}$& $2^{[2]}$&11 &15 &$\mathbf{351}$& $3^{[2]}$& $1^{[3]}$& $2^{[3]}$&12&16\\

\hline

$\mathbf{64}$&$3^{[3]}$& $1^{[3]}$& 2 &10 &14 &$\mathbf{160}$&$3^{[3]}$& $1^{[3]}$& $2^{[1]}$  & 11 &15&$\mathbf{256}$& $3^{[3]}$&$1^{[3]}$& $2^{[2]}$&12 &16 &$\mathbf{352}$& $3^{[3]}$& $1^{[3]}$& $2^{[3]}$&13&17\\

\hline
$\mathbf{65}$&2&3 & 1 &8 &8 &$\mathbf{161}$&2&3 & $1^{[1]}$&9& 9&$\mathbf{257}$&2&3 & $1^{[2]}$ & 10&10&$\mathbf{353}$&2&3 & $1^{[3]}$&11&11\\

\hline

$\mathbf{66}$&$2^{[1]}$&3 & 1 &9 &9 &$\mathbf{162}$&$2^{[1]}$&3 & $1^{[1]}$&10&10 &$\mathbf{258}$&$2^{[1]}$&3 & $1^{[2]}$ &11 &11&$\mathbf{354}$&$2^{[1]}$&3 & $1^{[3]}$&12&12\\

\hline

$\mathbf{67}$&$2^{[2]}$&3 & 1 & 10&10 &$\mathbf{163}$&$2^{[2]}$&3 & $1^{[1]}$&11& 11&$\mathbf{259}$&$2^{[2]}$&3 & $1^{[2]}$ &12 &12&$\mathbf{355}$&$2^{[2]}$&3 & $1^{[3]}$&13&13\\

\hline

$\mathbf{68}$&$2^{[3]}$&3 & 1 &11 &11 &$\mathbf{164}$&$2^{[3]}$&3 & $1^{[1]}$&12&12 &$\mathbf{260}$&$2^{[3]}$&3 & $1^{[2]}$ & 13&13&$\mathbf{356}$&$2^{[3]}$&3 & $1^{[3]}$&14&14\\
\hline

$\mathbf{69}$&2&$3^{[1]}$ & 1 & 5&9 &$\mathbf{165}$&2&$3^{[1]}$ & $1^{[1]}$&14&10 &$\mathbf{261}$&2&$3^{[1]}$ & $1^{[2]}$ & 15&11&$\mathbf{357}$&2&$3^{[1]}$ & $1^{[3]}$&16&12\\

\hline

$\mathbf{70}$&$2^{[1]}$&$3^{[1]}$ & 1 &2 & 10&$\mathbf{166}$&$2^{[1]}$&$3^{[1]}$ & $1^{[1]}$&11&11 &$\mathbf{262}$&$2^{[1]}$&$3^{[1]}$ & $1^{[2]}$ & 12&12&$\mathbf{358}$&$2^{[1]}$&$3^{[1]}$ & $1^{[3]}$&13&13\\

\hline

$\mathbf{71}$&$2^{[2]}$&$3^{[1]}$ & 1 & 3&11 &$\mathbf{167}$&$2^{[2]}$&$3^{[1]}$ & $1^{[1]}$&12&12 &$\mathbf{263}$&$2^{[2]}$&$3^{[1]}$ & $1^{[2]}$ & 13&13&$\mathbf{359}$&$2^{[2]}$&$3^{[1]}$ & $1^{[3]}$&14&14\\

\hline

$\mathbf{72}$&$2^{[3]}$&$3^{[1]}$ & 1 &4 & 12&$\mathbf{168}$&$2^{[3]}$&$3^{[1]}$ & $1^{[1]}$&13& 13&$\mathbf{264}$&$2^{[3]}$&$3^{[1]}$ & $1^{[2]}$ & 14&14&$\mathbf{360}$&$2^{[2]}$&$3^{[1]}$ & $1^{[3]}$&15&15\\

\hline

$\mathbf{73}$&2&$3^{[2]}$ & 1 &6 &10 &$\mathbf{169}$&2&$3^{[2]}$ & $1^{[1]}$&7&11 &$\mathbf{265}$&2&$3^{[2]}$ & $1^{[2]}$ &16 &12&$\mathbf{361}$&2&$3^{[2]}$ & $1^{[3]}$&17&13\\

\hline

$\mathbf{74}$&$2^{[1]}$&$3^{[2]}$ & 1 &7 &11 &$\mathbf{170}$&$2^{[1]}$&$3^{[2]}$ & $1^{[1]}$&8& 12&$\mathbf{266}$&$2^{[1]}$&$3^{[2]}$ & $1^{[2]}$ & 17&13&$\mathbf{362}$&$2^{[1]}$&$3^{[2]}$ & $1^{[3]}$&18&14\\
\hline

$\mathbf{75}$&$2^{[2]}$&$3^{[2]}$ & 1 & 4& 12&$\mathbf{171}$&$2^{[2]}$&$3^{[2]}$ & $1^{[1]}$&5& 13&$\mathbf{267}$&$2^{[2]}$&$3^{[2]}$ & $1^{[2]}$ & 14&14&$\mathbf{363}$&$2^{[2]}$&$3^{[2]}$ & $1^{[3]}$&15&15\\

\hline

$\mathbf{76}$&$2^{[3]}$&$3^{[2]}$ & 1 &5 &13 &$\mathbf{172}$&$2^{[3]}$&$3^{[2]}$ & $1^{[1]}$&6&14 &$\mathbf{268}$&$2^{[3]}$&$3^{[2]}$ & $1^{[2]}$ & 15&15&$\mathbf{364}$&$2^{[3]}$&$3^{[2]}$ & $1^{[3]}$&16&16\\

\hline

$\mathbf{77}$&2&$3^{[3]}$ & 1 &7 & 11&$\mathbf{173}$&2&$3^{[3]}$ & $1^{[1]}$&8& 12&$\mathbf{269}$&2&$3^{[3]}$ & $1^{[2]}$ & 9&13&$\mathbf{365}$&2&$3^{[3]}$ & $1^{[3]}$&18&14\\

\hline

$\mathbf{78}$&$2^{[1]}$&$3^{[3]}$ & 1 & 8&12 &$\mathbf{174}$&$2^{[1]}$&$3^{[3]}$ & $1^{[1]}$&9&13 &$\mathbf{270}$&$2^{[1]}$&$3^{[3]}$ & $1^{[2]}$ & 10&14&$\mathbf{366}$&$2^{[1]}$&$3^{[3]}$ & $1^{[3]}$&19&15\\

\hline

$\mathbf{79}$&$2^{[2]}$&$3^{[3]}$ & 1 &9 &13 &$\mathbf{175}$&$2^{[2]}$&$3^{[3]}$ & $1^{[1]}$&10&14 &$\mathbf{271}$&$2^{[2]}$&$3^{[3]}$ & $1^{[2]}$ & 11&15&$\mathbf{367}$&$2^{[2]}$&$3^{[3]}$ & $1^{[3]}$&20&16\\

\hline

$\mathbf{80}$&$2^{[3]}$&$3^{[3]}$ & 1 &6 &14 &$\mathbf{176}$&$2^{[3]}$&$3^{[3]}$ & $1^{[1]}$&7& 15&$\mathbf{272}$&$2^{[3]}$&$3^{[3]}$ & $1^{[2]}$ & 8&16&$\mathbf{368}$&$2^{[3]}$&$3^{[3]}$ & $1^{[3]}$&17&17\\

\hline

$\mathbf{81}$&3&2 & 1 &12 &12 &$\mathbf{177}$&3&2 & $1^{[1]}$&13& 13&$\mathbf{273}$&3&2 & $1^{[2]}$ & 14&14&$\mathbf{369}$&3&2 & $1^{[3]}$&15&15\\

\hline

$\mathbf{82}$&$3^{[1]}$&2 & 1 &9 &13 &$\mathbf{178}$&$3^{[1]}$&2 & $1^{[1]}$&10&14&$\mathbf{274}$&$3^{[1]}$&2 & $1^{[2]}$ & 11&15&$\mathbf{370}$&$3^{[1]}$&2 & $1^{[3]}$&12&16\\

\hline

$\mathbf{83}$&$3^{[2]}$&2 & 1 &10 & 14&$\mathbf{179}$&$3^{[2]}$&2 & $1^{[1]}$&11&15 &$\mathbf{275}$&$3^{[2]}$&2 & $1^{[2]}$ & 12&16&$\mathbf{371}$&$3^{[2]}$&2 & $1^{[3]}$&13&17\\

\hline

$\mathbf{84}$&$3^{[3]}$&2 & 1 &11 &15 &$\mathbf{180}$&$3^{[3]}$&2 & $1^{[1]}$&12&16 &$\mathbf{276}$&$3^{[3]}$&2 & $1^{[2]}$ & 13&17&$\mathbf{372}$&$3^{[3]}$&2 & $1^{[3]}$&14&18\\

\hline

$\mathbf{85}$&3&$2^{[1]}$ & 1 & 5&13 &$\mathbf{181}$&3&$2^{[1]}$ & $1^{[1]}$&14&14 &$\mathbf{277}$&3&$2^{[1]}$ & $1^{[2]}$ &15 &15&$\mathbf{373}$&3&$2^{[1]}$ & $1^{[3]}$&16&16\\

\hline

$\mathbf{86}$&$3^{[1]}$&$2^{[1]}$ & 1 &6 &14 &$\mathbf{182}$&$3^{[1]}$&$2^{[1]}$ & $1^{[1]}$&15& 15&$\mathbf{278}$&$3^{[1]}$&$2^{[1]}$ & $1^{[2]}$ & 16&16&$\mathbf{374}$&$3^{[1]}$&$2^{[1]}$ & $1^{[3]}$&17&17\\
\hline

$\mathbf{87}$&$3^{[2]}$&$2^{[1]}$ & 1 &3 & 15&$\mathbf{183}$&$3^{[2]}$&$2^{[1]}$ & $1^{[1]}$&12&16 &$\mathbf{279}$&$3^{[2]}$&$2^{[1]}$ & $1^{[2]}$ & 13&17&$\mathbf{375}$&$3^{[2]}$&$2^{[1]}$ & $1^{[3]}$&14&18\\

\hline

$\mathbf{88}$&$3^{[3]}$&$2^{[1]}$ & 1 &4 &16 &$\mathbf{184}$&$3^{[3]}$&$2^{[1]}$ & $1^{[1]}$&13& 17&$\mathbf{280}$&$3^{[3]}$&$2^{[1]}$ & $1^{[2]}$ & 14&18&$\mathbf{376}$&$3^{[3]}$&$2^{[1]}$ & $1^{[3]}$&15&19\\

\hline

$\mathbf{89}$&3&$2^{[2]}$ & 1 & 6&14 &$\mathbf{185}$&3&$2^{[2]}$ & $1^{[1]}$&7&15 &$\mathbf{281}$&3&$2^{[2]}$ & $1^{[2]}$ &16 &16&$\mathbf{377}$&3&$2^{[2]}$ & $1^{[3]}$&17&17\\

\hline

$\mathbf{90}$&$3^{[1]}$&$2^{[2]}$ & 1 &7 &15 &$\mathbf{186}$&$3^{[1]}$&$2^{[2]}$ & $1^{[1]}$&8& 16&$\mathbf{282}$&$3^{[1]}$&$2^{[2]}$ & $1^{[2]}$ & 17&17&$\mathbf{378}$&$3^{[1]}$&$2^{[2]}$ & $1^{[3]}$&18&18\\

\hline

$\mathbf{91}$&$3^{[2]}$&$2^{[2]}$ & 1 &8 &16 &$\mathbf{187}$&$3^{[2]}$&$2^{[2]}$ & $1^{[1]}$&9&17 &$\mathbf{283}$&$3^{[2]}$&$2^{[2]}$ & $1^{[2]}$ & 18&18&$\mathbf{379}$&$3^{[2]}$&$2^{[2]}$ & $1^{[3]}$&19&19\\
\hline

$\mathbf{92}$&$3^{[3]}$&$2^{[2]}$ & 1 &5 &17 &$\mathbf{188}$&$3^{[3]}$&$2^{[2]}$ & $1^{[1]}$&6&18 &$\mathbf{284}$&$3^{[3]}$&$2^{[2]}$ & $1^{[2]}$ & 15&19&$\mathbf{380}$&$3^{[3]}$&$2^{[2]}$ & $1^{[3]}$&16&20\\

\hline

$\mathbf{93}$&3&$2^{[3]}$ & 1 &7 &15 &$\mathbf{189}$&3&$2^{[3]}$ & $1^{[1]}$&8&16 &$\mathbf{285}$&3&$2^{[3]}$ & $1^{[2]}$ & 9&17&$\mathbf{381}$&3&$2^{[3]}$ & $1^{[3]}$&18&18\\

\hline

$\mathbf{94}$&$3^{[1]}$&$2^{[3]}$ & 1 & 8&16 &$\mathbf{190}$&$3^{[1]}$&$2^{[3]}$ & $1^{[1]}$&9&17 &$\mathbf{286}$&$3^{[1]}$&$2^{[3]}$ & $1^{[2]}$ & 10&18&$\mathbf{382}$&$3^{[1]}$&$2^{[3]}$ & $1^{[3]}$&19&19\\

\hline

$\mathbf{95}$&$3^{[2]}$&$2^{[3]}$ & 1 &9 &17 &$\mathbf{191}$&$3^{[2]}$&$2^{[3]}$ & $1^{[1]}$&10&18 &$\mathbf{287}$&$3^{[2]}$&$2^{[3]}$ & $1^{[2]}$ & 11&19&$\mathbf{383}$&$3^{[2]}$&$2^{[3]}$ & $1^{[3]}$&20&20\\
\hline

$\mathbf{96}$&$3^{[3]}$&$2^{[3]}$ & 1 &10 &18 &$\mathbf{192}$&$3^{[3]}$&$2^{[3]}$ & $1^{[1]}$&11&19 &$\mathbf{288}$&$3^{[2]}$&$2^{[3]}$ & $1^{[2]}$ & 12&20&$\mathbf{384}$&$3^{[3]}$&$2^{[3]}$ & $1^{[3]}$&21&21\\

\hline \hline

\end{tabular}}
\end{center}
\end{table}

%\section{Example Section}
%\label{sec1}
%% Labels are used to cross-reference an item using \ref command.

%Section text. See Subsection \ref{subsec1}.

%% Use \subsection commands to start a subsection.
%\subsection{Example Subsection}
%\label{subsec1}

\end{document}